\documentclass[leqno,11pt,a4paper,twoside]{article}%
\usepackage{amssymb}
\usepackage{amsfonts}
\usepackage{amsmath}
\usepackage{graphicx}
\usepackage{color}
\usepackage{amscd}%
\providecommand{\U}[1]{\protect\rule{.1in}{.1in}}
\newtheorem{theorem}{Theorem}

\newenvironment{assumption}[1][Assumption]{\noindent\textbf{#1.} }{\ }

\newtheorem{corollary}[theorem]{Corollary}

\newtheorem{example}[theorem]{Example}

\newtheorem{lemma}[theorem]{Lemma}

\newtheorem{proposition}[theorem]{Proposition}
\newtheorem{remark}[theorem]{Remark}

\newenvironment{proof}[1][Proof]{\noindent\textbf{#1.} }{\ \rule{0.5em}{0.5em}}
\begin{document}

\title{Regularity and Stability for the Semigroup of Jump Diffusions with
State-Dependent Intensity }
\author{Vlad Bally\thanks{Universit\'{e} Paris-Est, LAMA (UMR 8050), UPEMLV, UPEC,
CNRS, F-77454, Marne-la-Vall\'{e}e, France, email: vlad.bally@u-pem.fr}
\and Dan Goreac\thanks{Universit\'{e} Paris-Est, LAMA (UMR 8050), UPEMLV, UPEC,
CNRS, F-77454, Marne-la-Vall\'{e}e, France, \textbf{Corresponding author},
email: dan.goreac@u-pem.fr}
\and Victor Rabiet\thanks{Universit\'{e} Paris-Est, LAMA (UMR 8050), UPEMLV, UPEC,
CNRS, F-77454, Marne-la-Vall\'{e}e, France}}
\maketitle

\begin{abstract}
We consider stochastic differential systems driven by a Brownian motion and a
Poisson point measure where the intensity measure of jumps depends on the
solution. This behavior is natural for several physical models (such as
Boltzmann equation, piecewise deterministic Markov processes, etc). First, we
give sufficient conditions guaranteeing that the semigroup associated with
such an equation preserves regularity by mapping the space of the of $k$-times
differentiable bounded functions into itself. Furthermore, we give an explicit
estimate of the operator norm. This is the key-ingredient in a quantitative
Trotter-Kato-type stability result: it allows us to give an explicit estimate
of the distance between two semigroups associated with different sets of
coefficients in terms of the difference between the corresponding
infinitesimal operators. As an application, we present a method allowing to
replace "small jumps" by a Brownian motion or by a drift component. The
example of the 2D Boltzmann equation is also treated in all detail.


\end{abstract}

\section{Introduction}

\nocite{Kurtz77}\nocite{EthierKurtz1986}\nocite{Walsh1986}%
\nocite{BallyFournier2011}\nocite{JacodShiryaev03}We propose a quantitative
analysis of the regularity of semigroups of operators associated with hybrid
piecewise-diffusive systems
\begin{equation}
\left.
\begin{array}
[c]{l}%
X_{t}=x+\sum_{l=1}^{\infty}\int_{0}^{t}\sigma_{l}(s,X_{s})dB_{s}^{l}+\int
_{0}^{t}b(s,X_{s})ds\\
\text{ \ \ \ \ \ }+\underset{\left[  0,t\right]  \times E\times%
\mathbb{R}
_{+}}{\int}c(s,z,X_{s-})1_{\{u\leq\gamma(s,z,X_{s-})\}}N_{\mu}(ds,dz,du).
\end{array}
\right.  \label{I1}%
\end{equation}
taking their values in some Euclidian space $%
\mathbb{R}
^{d}$. Here,

\begin{itemize}
\item $\left(  E,\mathcal{E}\right)  $ is a measurable space,

\item $N_{\mu}(ds,dz,du)$ is a homogenous Poisson point measure on
$E\times(0,\infty)$ with intensity measure $\mu(dz)\times1_{(0,\infty)}(u)du,$

\item $W_{t}=(W_{t}^{l})_{l\in\mathbb{N}}$ is an infinite-dimensional Brownian
motion (independent of $N_{\mu})$\ and

\item the coefficients $\sigma_{l},b:\mathbb{R}_{+}\times\mathbb{R}%
^{d}\rightarrow\mathbb{R}^{d}$ and $c:\mathbb{R}_{+}\times E\times
\mathbb{R}^{d}\rightarrow\mathbb{R}^{d},$ $\gamma:\mathbb{R}_{+}\times
E\times\mathbb{R}^{d}\rightarrow\lbrack0,\infty)$ are assumed to be smooth enough.
\end{itemize}

Whenever the jump intensity $\gamma$ is constant, one deals with classical
stochastic differential systems with jumps. Regularity of the associated flow
is then immediate (see \cite{IkedaWatanabe1989} or \cite{Kunita2004}).
However, if $\gamma$ is non-constant, the position $X_{s-}$ of the solution
plays an important part in the intensity of jumps. This latter framework
occurs in a wide variety \ of applications and it will receive our attention
throughout the paper.

Our first result (see Theorem \ref{M}), which is the core of the paper,
consists in proving that, under natural assumptions, the semigroup
$\mathcal{P}_{t}f(x)=\mathbb{E}(f(X_{t}(x))$ propagates regularity in finite
time $T>0$, i.e.%
\begin{equation}
\sup_{t\leq T}\left\Vert \mathcal{P}_{t}f\right\Vert _{k,\infty}\leq
Q_{k}(T,\mathcal{P})\left\Vert f\right\Vert _{k,\infty},\quad\forall f\in
C_{b}^{k}(R^{d}). \label{I1a}%
\end{equation}
Here, $\left\Vert f\right\Vert _{k,\infty}$ is the infinite norm of $f$ and
its first $k$ derivatives. \ In the case $k=0,$ this means that $\mathcal{P}%
_{t}$ is a Feller semigroup.

As we have already hinted, the main difficulty to overcome is due to the
presence of the jump intensity $\gamma(s,z,X_{s-})$. In classical jump
equation, the indicator function $1_{\{u\leq\gamma(s,z,X_{s-})\}}$ does not
appear and one may construct a version of the solution such that $x\rightarrow
X_{t}(x)$ is $k$-times differentiable (see \cite{Kunita2004}). Next one
proceeds with differentiating the associated semigroup and using chain rule
$\partial_{x_{i}}\mathcal{P}_{t}f(x)=\sum_{j=1}^{d}\mathbb{E}\left[
\partial_{j}f(X_{t}(x))\right]  \partial_{x_{i}}X_{t}^{j}(x)$ and concludes
that (\ref{I1a}) holds for $k=1$ with $Q_{1}(T,P)=\underset{t\leq T}{\sup
}\underset{x}{\text{ }\sup}\mathbb{E}\left[  \left\vert \nabla X_{t}%
(x)\right\vert \right]  .$ However, in our framework, the indicator function
is present such that the stochastic differential representation of the
solution in (\ref{I1}) is not appropriate. We will employ the alternative
representation in (\ref{D1}) (known in the engineering literature as "real
shock" representation, whileas (\ref{I1}) is known as the "fictive shock"
representation). The specificity of our framework is that the law of the jumps
depends on the trajectory and this dependence is quantified by $\gamma.$ As
consequence, the constants $Q_{k}$ will depend on some quantities of type
$\int_{E}\left\vert \partial^{\alpha}\ln\gamma\left(  t,z,x\right)
\right\vert ^{p}\gamma\left(  t,z,x\right)  \mu\left(  dz\right)  $ (for
appropriate $p\leq k$ and index $\alpha$; the presence of such terms is
inspired by Malliavin calculus techniques).

A second result is a stability propety in the line of the Trotter - Kato
theorem (cf. \cite[Theorem 4.4]{Pazy1983}). We consider a sequence $\left(
\mathcal{P}_{t}^{n}\right)  _{n\in\mathbb{N}}$ of semigroups of operators with
generators $\mathcal{L}^{n}$ and we assume that, for some $q\in\mathbb{N},$
\begin{equation}
\left\Vert (\mathcal{L}^{n}-\mathcal{L})f\right\Vert _{\infty}\leq
\varepsilon\times\left\Vert f\right\Vert _{q,\infty},\quad\text{for all }f\in
C_{b}^{q}(\mathbb{R}^{d}). \label{I1b}%
\end{equation}
Here, $\mathcal{L}$ stands for the infinitesimal operator associated with
(\ref{I1}). In Theorem \ref{Conv} we prove that, under suitable hypotheses,
the previous inequality yields%
\begin{equation}
\left\Vert (\mathcal{P}_{t}^{n}-\mathcal{P}_{t})f\right\Vert _{\infty}%
\leq\varepsilon\times Q_{q}(T,\mathcal{P)\times}\left\Vert f\right\Vert
_{q,\infty},\quad\text{for all }f\in C_{b}^{q}(\mathbb{R}^{d}). \label{I1c}%
\end{equation}
In order to undersand the link betwen this result and the property
(\ref{I1a}), one writes%
\[
\mathcal{P}_{t}f(x)-\mathcal{P}_{t}^{n}f(x)=\int_{0}^{t}\partial
_{s}\mathcal{P}_{t-s}^{n}\mathcal{P}_{s}f(x)ds=\int_{0}^{t}\mathcal{P}%
_{t-s}^{n}(\mathcal{L}^{n}-\mathcal{L})\mathcal{P}_{s}f(x)ds
\]
and notice that by (\ref{I1b}) first and by (\ref{I1a}) next%
\[
\left\Vert \mathcal{P}_{t-s}^{n}(\mathcal{L}^{n}-\mathcal{L})\mathcal{P}%
_{s}f\right\Vert _{\infty}\leq\left\Vert (\mathcal{L}^{n}-\mathcal{L}%
)\mathcal{P}_{s}f\right\Vert _{\infty}\leq\varepsilon\times\left\Vert
\mathcal{P}_{s}f\right\Vert _{q,\infty}\leq\varepsilon\times Q_{q}%
(T,\mathcal{P)}\left\Vert f\right\Vert _{q,\infty}.
\]
We finally mention that in the paper we deal with non-homogenous semigroups
and the inequalities are written with weighted norms (for simplicity we have
chosen to present the results with usual inifinity norms in this introduction).

If $\mu$ is a finite measure and $\sigma$ is null, the solution of the above
equation (\ref{I1}) relates to the class of Piecewise Deterministic Markov
Process (abridged PDMP). These equations have been introduced in
\cite{Davis_84} and studied in detail in \cite{davis_93}. A wide literature is
available on the subject of PDMP as they present an increasing amount of
applications: on/off systems (cf. \cite{Boxma_Kaspi_Kella_Perry_2005}),
reliability (e.g. \cite{DufourDutuitGonzalesZhang2008}), simulations and
approximations of reaction networks (e.g. \cite{Gillespie1976},
\cite{AlfonsiCancesTurinici2005}, \cite{CruduDebussheRadulescu2012}, with some
error bounds hinted at in \cite{JahnkeKreim2012} or
\cite{GangulyAltinanKoeppl}), neuron models (e.g. \cite{BuckwarRiedler},
\cite{BressloffNewby2013}), etc. The reader may equally take a look at the
recent book \cite{Cocozza20+} for an overview of some applications. In
engineering community, these equations are also known as "transport equations"
(see \cite{LapeyrePardouxSentis1998} or \cite{Kolokoltsov2011}).

To the best of our knowledge, in the general case (including a diffusion
component and an infinite number of jumps), under suitable assumptions, the
first proof of existence and uniqueness of the solution of equation (\ref{I1})
is given in \cite{Graham1992}.

Now assume that, for one purpose or another, one aims at applying some kind of
numerical algorithm in order to simulate the solution of equation (\ref{I1}).
Furthermore, assume for the moment, that $\sigma=b=0$ such that%
\begin{equation}
X_{t}=x+\int_{0}^{t}\int_{E}\int_{(0,\infty)}c(s,z,X_{s-})1_{\{u\leq
\gamma(s,z,X_{s-})\}}N_{\mu}(ds,dz,du).
\end{equation}
If $\mu(E)$ is finite, then one deals with a finite number of jumps in any
interval of time, such that the solution $X$ is given with respect to a
compound Poisson process that can be explicitly simulated (leading, in
particular chemistry-inspired settings, to what is commonly known as
Gillespie's algorithm \cite{Gillespie1976}; for other general aspects on
simulation, see also \cite{LapeyrePardouxSentis1998}). However, even in this
rather smooth case, the presence of a trajectory-triggered jump (i.e.
dependence on $x$ in the jump intensity $\gamma$) can lead, in certain regions
(as $\gamma$ gets large) to the accumulation of many (possibly) small jumps.
In this case, the algorithm becomes very slow. One way of dealing with the
problem is to replace these small jumps with an averaged motion leading
(piecewise) to an ordinary differential equation (e.g. in
\cite{AlfonsiCancesTuriniciDiVentura2004Rep}). Within the context of reaction
networks, some intuitions on the partition of reactions and species to get the
hybrid behavior as well as qualitative behavior (convergence to PDMP) are
specified, for example, in \cite{CruduDebussheRadulescu2012}. Further
heuristics can be found in \cite{AngiusBalboBeccuti2015}.

In the general framework of infinite $\mu(E),$ this direct approach may fail
to provide fast solutions (except particular situations e.g. in
\cite{ProtterTalay1997}). To provide an answer, the natural idea is to
truncate the "small jumps" on some compatible family of sets $\left(
E_{n}\right)  _{n\in\mathbb{N}}$ and simulating $X_{t}^{n}$ solution of
\begin{equation}
X_{t}^{n}=x+\int_{0}^{t}\int_{E_{n}^{c}}\int_{(0,\infty)}c_{n}(s,z,X_{s-}%
^{n})1_{\{u\leq\gamma(s,z,X_{s-}^{n})\}}N_{\mu}(ds,dz,du). \label{I1d}%
\end{equation}
This procedure leads to a large error. To improve it, one might want to
further replace the "small jumps" from $E_{n}$ by a Brownian diffusion term
leading to%
\begin{equation}
\left.
\begin{array}
[c]{l}%
X_{t}^{n}=x+\int_{0}^{t}\int_{E_{n}}\sigma_{n}(s,z,X_{s-}^{n})W_{\mu
}(ds,dz)+\int_{0}^{t}b_{n}(s,X_{s}^{n})ds\\
\text{ \ \ \ \ }+\int_{0}^{t}\int_{E_{n}^{c}}\int_{(0,\infty)}c_{n}%
(s,z,X_{s-}^{n})1_{\{u\leq\gamma_{n}(s,z,X_{s-}^{n})\}}N_{\mu}(ds,dz,du)
\end{array}
\right.  \label{I1e}%
\end{equation}
where $W_{\mu}$ is a time-space Gaussian random measure (associated with
$\mathbb{L}^{2}\left(  \mu\right)  $; standard procedure allows interpreting
$W_{\mu}$ as in equation (\ref{I1})). The specific form of $\sigma_{n}$ and
$b_{n}$ is obtained by using a second-order Taylor development in the
infinitesimal operator of the initial equation.

This idea goes back to \cite{AsmussenRosinski2001}. In the case of systems
driven by a L\'{e}vy process (with $\gamma$ fixed), \cite{Fournier2012} gives
a precise estimate of the error and compares the approximation obtained by
truncation as in equation (\ref{I1d}) with the one obtained by adding a
Gaussian noise as in equation (\ref{I1e}). An enlightening discussion on the
complexity of the two methods is also provided. Similar results concerning
Kac's equation are obtained in \cite{FournierGodinho2012} and for a
Boltzmann-type equation in \cite{GODINHO2013}. For some recent developpement
on asymptotics of Boltzmann-type equation, we also mention \cite{He2014}.
Finally, it is worth mentioning that the converse approach (replacing Brownian
with jump diffusions) may also be useful. The engineering literature is quite
abundant in overviews of numerical methods for (continuous) diffusion
processes using jump-type schemes. In this case the stochastic integral with
respect to the Brownian motion is replaced by an integral with respect to a
jump process.

The aim of the present paper is to provide quantitative estimates of the weak
approximation error when substituting the original system (\ref{I1}) with
hybrid (piecewise diffusive Markov system (\ref{I1e})) in the general case
when $\gamma$ is trajectory-dependent (which constitutes the main difficulty
to overcome). At intuitive level, Trotter-Kato-type results (cf. \cite[Theorem
4.4]{Pazy1983}) give the qualitative behavior. If $\mathcal{P}^{n}$ (resp.
$\mathcal{L}^{n}$) is the semigroup (resp. infinitesimal generator) associated
with (\ref{I1e}) and $\mathcal{P}$ (resp. $\mathcal{L}$) is the semigroup
(resp. infinitesimal generator) associated with (\ref{I1}), under Feller-type
conditions, convergence of $\mathcal{L}^{n}$ to $\mathcal{L}$ will imply the
corresponding convergence of semigroups. This type of qualitative behavior can
be found, for instance, in \cite{Kurtz71} (leading to drift),
\cite{BallKurtzPopovic2006} (leading to piecewise diffusive processes). In
order to get error bounds (leading to a quantitative estimate), one employes
(\ref{I1c}).

A somewhat different motivation for our work comes from a method introduced in
\cite{BenaimHirsch1996} (see also \cite{BenaimBouguetCloez2016}) to study
convergence to equilibrium for Markov chains. Roughly speaking, instead of
looking into the long-time behavior of the Markov chain $Y_{n},n\in
\mathbb{N},$ one replaces this chain by a Markov process $X_{t}$ sharing the
same asymptotics ($t\mapsto X_{t}$ being an "asymptotic pseudotrajectory"). In
\cite{Locherbach2017}, the results of our paper are used in order to extend
this method (of \cite{BenaimHirsch1996}) to the piecewise deterministic Markov
framework$.$

Finally, although many biological intuitions exist on the use of hybrid models
for reaction systems (e.g. \cite{AngiusBalboBeccuti2015}), the quantitative
estimates in our paper may turn out to provide a (purely mathematical)
selection criterion for the components to be averaged and the contributions to
be kept within the jump component. The use of diffusions punctuated by jumps
(as mesoscopic approach) responds, on one hand, to the question of speeding up
algorithms and, on the other, of keeping a high degree of stochasticity
(needed, for example, to exhibit multistable regimes).

This paper is organized as follows. We begin with presenting the main
notations used throughout the paper. We proceed, in Section \ref{SectPrelim}
with the main elements leading to the processes involved. First, we recall
some classical results on cylindrical diffusion-driven processes and the
regularity of the induced flow (Section \ref{SectContDiff}). Next, in Section
\ref{SectJumpMech}, we introduce the jumping mechanism as well as the standing
assumptions. We proceed with the construction of hybrid systems (piecewise
diffusive with trajectory-triggered jumps) in Section \ref{SectHybridSyst}. We
begin with some localization estimates when the underlying measure is finite
in Lemma \ref{Distance}. We also recall some elements on fictive and real
shocks leading to some kind of Marked-point process representation of our
system. These elements turn out to be of particular importance in providing
the differentiability of the flow generated by our hybrid system. The norm
notations and the $\mathbb{L}^{p}$-regularity of the solution (uniformly with
respect to the initial data) are given in Section \ref{SectDifferentiability}.

The differentiability of the associated semigroup is studied in Section
\ref{SectionDiffSemigroup} (with the main result being Theorem
\ref{ThDiffSemigroup1} whose uniform estimates extend to general underlying
measures in Theorem \ref{M}).

Section \ref{SectionDistanceSemigroups} gives quantitative results on the
distance between semigroups associated with such systems. The natural
assumptions are presented in the first subsection. The main result Theorem
\ref{Conv} produces quantitative upper-bounds for the distance between
semigroups starting from the distance between infinitesimal operators.

We present two classes of applications. In Section \ref{Section3Regimes}, we
imagine a piecewise deterministic Markov process presenting three regimes and
leading to a hybrid approximation with explicit distance on associated
semigroups. First, we provide a theoretical framework describing the model,
the regimes, the assumptions and the main qualitative behavior (in Theorem
\ref{ThExp}). Next, explicit measures make the object of a simple example to
which our result is applied.

The second class of examples is given by a two-dimensional Boltzmann equation
(following the approach in \cite{BallyFournier2011}) in Section
\ref{SectionBoltzmann}. We begin with describing the model, its probabilistic
interpretation and the (cut off) approximation given in
\cite{BallyFournier2011} and leading to a pure-jump PDMP. In this approximated
model, using our results, we replace small jumps with either a drift term
(first-order approximation provided in Theorem \ref{ThBoltzmannOrder1}) or a
diffusion term (second-order approximation provided in Theorem
\ref{ThBoltzmann2}).

\section{Notations}

Let $(E,\mathcal{E})$ be a measurable space and $\mu$ be a (fixed) non
negative $\sigma$-finite measure on $(E,\mathcal{E}).$

\begin{itemize}
\item Given a standard Euclidian state-space $\mathbb{R}^{m},$ the spaces
$\mathbb{L}^{p}(\mu)$ (for $1\leq p\leq\infty$) will denote the usual space of
$p-$power integrable, $\mathbb{R}^{m}$-valued functions defined on $E$. This
space is endowed with the usual norm%
\[
\left\Vert \phi\right\Vert _{\mathbb{L}^{p}(\mu)}=\left(  \int_{E}\left\vert
\phi(z)\right\vert ^{p}\mu(dz)\right)  ^{\frac{1}{p}},
\]
for all measurable function $\phi:E\rightarrow\mathbb{R}^{m}.$ For notation
purposes and by abuse of notation, the dependence on $m$ is dropped (one
should write $\mathbb{L}^{2}(\mu;\mathbb{R}^{m})$). The norm $\left\vert
\cdot\right\vert $ denotes the classical, Euclidian norm on $\mathbb{R}^{m}.$

\item The space $C_{b}^{q}(\mathbb{R}^{m})$ is the space of real-valued
functions on $\mathbb{R}^{m}$ whose partial derivatives up to order $q$ are
bounded and continuous.
\end{itemize}

Given a (fixed) probability space $\left(  \Omega,\mathcal{F},\mathbb{P}%
\right)  $ and a (fixed) time horizon $T>0$,

\begin{itemize}
\item If $\xi$ is an $\mathbb{R}^{m}$-valued random variable on $\Omega,$ we
denote, as usual, $\left\Vert \xi\right\Vert _{p}=\left(  \mathbb{E}\left[
\left\vert \xi\right\vert ^{p}\right]  \right)  ^{\frac{1}{p}}.$

\item If $Y$ is an adapted real-valued process and $Z$ is an $\mathbb{L}%
^{2}(\mu)$-valued process, then we denote by%
\[
\left\Vert Y\right\Vert _{T,p}=\left(  \mathbb{E}\left[  \sup_{t\leq
T}\left\vert Y_{t}\right\vert ^{p}\right]  \right)  ^{\frac{1}{p}}\quad
and\quad\left\Vert Z\right\Vert _{T,p}=\left(  \mathbb{E}\left[  \sup_{t\leq
T}\left\Vert Z_{t}\right\Vert _{\mathbb{L}^{2}(\mu)}^{p}\right]  \right)
^{\frac{1}{p}}.
\]

\item We use $\mathcal{M}_{T}$ to denote the space of the measurable functions
$f:[0,T]\times E\times\mathbb{R}^{d}\rightarrow\mathbb{R}$ (where metric space
are endowed with usual Borel fields)$.$ For $f\in\mathcal{M}_{T},$ we consider
the norm%
\begin{equation}
\left\Vert f\right\Vert _{(\mu,\infty)}=\sup_{t\leq T}\sup_{x\in R^{d}%
}\left\Vert f(t,\cdot,x)\right\Vert _{\mathbb{L}^{2}(\mu)}. \label{G1}%
\end{equation}

\item Similar norm can be induced on $\mathcal{M}_{T}^{d}$ by replacing
$\mathbb{L}^{2}(\mu;\mathbb{R})$ with $\mathbb{L}^{2}(\mu;\mathbb{R}^{d})$ norms.

\item For a multi-index $\alpha=(\alpha_{1},...,\alpha_{q})\in\{1,...,d\}^{q}$
we denote $\left\vert \alpha\right\vert =q$ the length of $\alpha$ and
$\partial_{x}^{\alpha}=\partial_{x_{\alpha_{1}}}....\partial_{x_{\alpha_{q}}}$
the corresponding derivative. To simplify notation, the variable $x$ may be
suppressed and we will use $\partial^{\alpha}$.

\item For $k\in\mathbb{N}^{\ast},$ we denote by $\mathbb{R}_{[k]}$ the family
of real-valued vectors indexed by multi-indexes of at most $k$ length i.e.
$\mathbb{R}_{[k]}=\{y_{[k]}=(y_{\beta})_{1\leq\left\vert \beta\right\vert \leq
k}:y_{\beta}\in\mathbb{R}\}$ and, for $y_{\left[  k\right]  }\in
\mathbb{R}_{[k]}$ we denote
\begin{equation}
\left\vert y_{[k]}\right\vert _{\mathbb{R}_{[k]}}=\sum_{1\leq\left\vert
\beta\right\vert \leq k}\left\vert y_{\beta}\right\vert ^{\frac{k+1}%
{\left\vert \beta\right\vert }}. \label{A6}%
\end{equation}
By convention, $\left\vert y_{[0]}\right\vert _{\mathbb{R}_{[0]}}=0.$
Similarly, $\mathbb{R}_{[k]}^{d}$ is defined for vectors whose components
belong to $\mathbb{R}^{d}$ and $\left\vert y_{\beta}\right\vert $ is then
computed with respect to the usual Euclidian norm on $\mathbb{R}^{d}$.

\item If $x\mapsto f(t,z,x)$ is a real-valued, $q$ times differentiable
function for every $(t,z)\in\lbrack0,T]\times E$ then, for every $1\leq l\leq
q$ we denote%
\begin{equation}
\left\Vert f\right\Vert _{l,q,(\mu,\infty)}=\sum_{l\leq\left\vert
\alpha\right\vert \leq q}\left\Vert \partial_{x}^{\alpha}f\right\Vert
_{(\mu,\infty)}\text{ and }\left\Vert f\right\Vert _{q,(\mu,\infty
)}=\left\Vert f\right\Vert _{(\mu,\infty)}+\left\Vert f\right\Vert
_{1,q,(\mu,\infty)} \label{G3}%
\end{equation}

\end{itemize}

\begin{remark}
We emphasize that in $\left\Vert f\right\Vert _{l,q,(\mu,\infty)},$ for
$l\geq1,$ only derivatives are involved ($\left\Vert f\right\Vert
_{(\mu,\infty)}$ itself does not appear).
\end{remark}

\begin{itemize}
\item For a measurable function $g:[0,T]\times\mathbb{R}^{d}\rightarrow
\mathbb{R}$ we denote by $\left\Vert g\right\Vert _{\infty}=\underset{\left(
t,x\right)  \in\lbrack0,T]\times\mathbb{R}^{d}}{\sup}\left\vert
g(t,x)\right\vert $ and, if $x\mapsto g(t,x)$ is $q$ times differentiable for
every $t\in\lbrack0,T],$ then
\[
\left\Vert g\right\Vert _{l,q,\infty}=\sum_{l\leq\left\vert \alpha\right\vert
\leq q}\left\Vert \partial_{x}^{\alpha}g\right\Vert _{\infty}\text{ and
}\left\Vert g\right\Vert _{q,\infty}=\left\Vert g\right\Vert _{\infty
}+\left\Vert g\right\Vert _{1,q,\infty}.
\]

\end{itemize}

\subsection{Preliminary results\label{SectPrelim}}

\subsubsection{Continuous Diffusion\label{SectContDiff}}

We assume the fixed probability space $\left(  \Omega,\mathcal{F}%
,\mathbb{P}\right)  $ to be endowed with a Gaussian noise $W_{\mu}$ based on
$\mu,$ as introduced by Walsh in \cite{Walsh1986}$.$ We recall that $W_{\mu}$
is a family of centred Gaussian random variables $W_{\mu}(t,h)$ indexed by
$\left(  t,h\right)  \in%
\mathbb{R}
_{+}\times\mathbb{L}^{2}(\mu)$ with covariances $\mathbb{E}\left[  W_{\mu
}(t,h)W_{\mu}(s,g)\right]  =(t\wedge s)\left\langle h,g\right\rangle
_{\mathbb{L}^{2}(\mu)}.$ Note that whenever $\left(  e_{l}\right)
_{l\in\mathbb{N}}\in\mathbb{L}^{2}(\mu)$ is an orthonormal basis, the family
$\left(  W_{\mu}(t,e_{l})\right)  _{l\in%
\mathbb{N}
}$ is a sequence of independent standard Brownian motions.

We briefly recall the stochastic integral with respect to $W_{\mu}$. One
considers the natural filtration $\mathcal{F}_{t}^{W}=\sigma(W_{\mu
}(s,h):s\leq t,h\in\mathbb{L}^{2}(\mu)),$ for all $t\geq0.$ For a process
$\phi:\mathbb{R}_{+}\times\Omega\rightarrow\mathbb{L}^{2}(\mu)$ which is
adapted (i.e $\left\langle \phi_{t},h\right\rangle _{\mathbb{L}^{2}(\mu)}$ is
$\mathcal{F}_{t}^{W}$ measurable for every $h\in\mathbb{L}^{2}(\mu)$) and for
which $\mathbb{E}\left[  \int_{0}^{T}\left\Vert \phi_{t}\right\Vert
_{L^{2}(\mu)}^{2}dt\right]  <\infty,$ for every $T>0$, one defines
\begin{equation}
\int_{0}^{t}\int_{E}\phi_{s}(z)W_{\mu}(ds,dz):=\sum_{l=1}^{\infty}\int_{0}%
^{t}\left\langle \phi_{s},e_{l}\right\rangle _{\mathbb{L}^{2}(\mu)}W_{\mu
}(ds,e_{l}). \label{G2}%
\end{equation}

Let $0\leq s\leq T$ be fixed. A non-homogeneous continuous diffusion process
$\Phi_{s,t}(x),s\leq t\leq T$ driven by $W_{\mu}$ with (regular) coefficients
$\sigma$ and $b$ is the solution of the stochastic equation%
\begin{equation}
\left.
\begin{array}
[c]{c}%
\Phi_{s,t}(x)=x+\int_{s}^{t}\int_{E}\sigma\left(  u,z,\Phi_{s,u}(x)\right)
W_{\mu}(du,dz)+\int_{s}^{t}b\left(  u,\Phi_{s,u}(x)\right)  du\\
=x+\sum_{l=1}^{\infty}\int_{0}^{t}\sigma_{l}\left(  u,\Phi_{s,u}(x)\right)
dB_{u}^{l}+\int_{s}^{t}b\left(  u,\Phi_{s,u}(x)\right)  du,
\end{array}
\right.  \label{G5}%
\end{equation}
with $B_{s}^{l}=W_{\mu}(s,e_{l})$ and $\sigma_{l}(u,x)=\left\langle
\sigma(u,\cdot,x),e_{l}\right\rangle _{\mathbb{L}^{2}(\mu)}.$

The following result is standard for finite-dimensional Brownian motions (e.g.
\cite{IkedaWatanabe1989}, \cite{Kunita2004}) and its generalization to this
setting is quite forward.

\begin{proposition}
Let us assume the following norm condition to hold true
\begin{equation}
\left\Vert \nabla\sigma\right\Vert _{(\mu,\infty)}+\left\Vert \nabla
b\right\Vert _{\infty}<\infty. \label{G5'}%
\end{equation}
Then, for every initial datum $x\in%
\mathbb{R}
^{d},$ the equation (\ref{G5}) has a unique strong solution. Moreover, if
$\left\Vert \sigma\right\Vert _{1,q+1,(\mu,\infty)}+\left\Vert b\right\Vert
_{1,q,\infty}<\infty,$ then there exists a version of this solution such that
$x\mapsto X_{s,t}(x)$ is $q$ times differentiable.
\end{proposition}

\begin{remark}
Let us note that the following (more usual) alternative representation for
this diffusion holds: let $a^{i,j}(t,x)=\int_{E}\sigma^{i}\sigma^{j}%
(t,z,x)\mu(dz),$ $1\leq i,j\leq d$ and set $\widehat{\sigma}=a^{\frac{1}{2}}.$
Then the law of $\Phi_{s,t}$ coincides with the law of $\widehat{\Phi}_{s,t}$
solution of
\[
\widehat{\Phi}_{s,t}(x)=x+\sum_{j=1}^{d}\int_{s}^{t}\widehat{\sigma}%
_{j}\left(  u,\widehat{\Phi}_{s,u}(x)\right)  dB_{u}^{j}+\int_{s}^{t}b\left(
u,\widehat{\Phi}_{s,u}(x)\right)  du,
\]
where $B=(B^{1},...,B^{d})$ is a standard Brownian motion. We prefer working
with the representation $\Phi_{s,t}$ (and not with $\widehat{\Phi}_{s,t}$) for
two reasons. First, the stochastic integral with respect to $W_{\mu}(du,dz)$
naturally appears in our problem. Moreover, if one liked to work with
$\widehat{\Phi}_{s,t},$ then one would have to compute $\widehat{\sigma}=$
$a^{\frac{1}{2}}$ and to derive regularity properties for $\widehat{\sigma}$
from regularity properties for $a,$ and this is more delicate (one needs some
ellipticity property for $a).$ In contrast, if one starts with the equation
(\ref{G5}), then the proof of the previous proposition is a straightforward
extension of the classical results.
\end{remark}

\subsubsection{Jump Mechanism and Further Notations\label{SectJumpMech}}

We assume the space $\Omega$ to be large enough to contain an independent
Poisson point measure on $E\times\mathbb{R}_{+}$ $\ $denoted by $N_{\mu}$ and
having a compensator $\widehat{N}_{\mu}(ds,dz,du)=ds\mu(dz)du.$ (For further
constructions and properties, the reader is referred to
\cite{IkedaWatanabe1989}). We just mention that, whenever $A_{l}\times
I_{l}\in\mathcal{E}\times\mathcal{B}(\mathbb{R}_{+}),$ $l=1,...,m$ are
disjoint sets, then $t\mapsto N_{\mu}(t,A_{l}\times I_{l})$ are independent
Poisson processes with parameters $\mu(A_{l})\times Leb(I_{l}).$ Here,
$\mathcal{B}(\mathbb{R}_{+})$ stands for the family of Borel subsets of
$\mathbb{R}_{+}$.

We consider now the coefficients $c\in\mathcal{M}_{T}^{d}$ and $\gamma
\in\mathcal{M}_{T}$ and we assume that there exist some functions
$l_{c},l_{\gamma}:E\rightarrow\mathbb{R}_{+}$ such that%
\begin{equation}
C_{\mu}(\gamma,c):=\sup_{t\leq T}\sup_{x\in R^{d}}\int_{E}(l_{\gamma
}(z)\left\vert c(t,z,x)\right\vert +l_{c}(z)\gamma(t,z,x))\mu(dz)<\infty
\label{Not7}%
\end{equation}
and such that, for every $x,y\in\mathbb{R}^{d},$ every $t\geq0$ and $z\in E,$%
\begin{equation}
\left\vert c(t,z,x)-c(t,z,y)\right\vert \leq l_{c}(z)\left\vert x-y\right\vert
,\quad\left\vert \gamma(t,z,x)-\gamma(t,z,y)\right\vert \leq l_{\gamma
}(z)\left\vert x-y\right\vert . \label{Not6}%
\end{equation}
Moreover, we assume that $\gamma$ takes non-negative values and
\begin{equation}
\Gamma:=\sup_{t\leq T}\sup_{x\in\mathbb{R}^{d}}\sup_{z\in E}\gamma
(t,z,x)<\infty. \label{Not5}%
\end{equation}
We also set, for any Borel set $G\subset E,$%
\begin{equation}
\alpha(G):=\sup_{t\leq T}\sup_{x\in R^{d}}\int_{G}\left\vert
c(t,z,x)\right\vert \gamma(t,z,x)\mu(dz) \label{Not8}%
\end{equation}
and assume that $\alpha(E)<\infty$.

\subsection{The Hybrid System\label{SectHybridSyst}}

We are interested in the (hybrid) stochastic differential equation%
\begin{equation}
\left.
\begin{array}
[c]{c}%
X_{s,t}(x)=x+\int_{s}^{t}\int_{E}\sigma(r,z,X_{s,r}(x))W_{\mu}(dr,dz)+\int
_{s}^{t}b(r,X_{s,r}(x)dr\\
+\int_{s}^{t}\int_{E\times\lbrack0,2\Gamma]}c(r,z,X_{s,r-}(x))1_{\{u\leq
\gamma(r,z,X_{s,r-}(x))\}}N_{\mu}(dr,dz,du).
\end{array}
\right.  \label{Not8a}%
\end{equation}

\begin{remark}
The stochastic components $W_{\mu}$ and $N_{\mu}$ are assumed to be associated
with the same measurable space $(E,\mathcal{E},\mu).$ This assumption is made
in order to avoid heavy notation. Alternatively, one may consider $W_{\mu}$ on
$(E,\mathcal{E},\mu)$ and $N_{\nu}$ on some (independent) space
$(F,\mathcal{F},\nu).$ For most examples, the space $E=\{1,...,d\}$ and the
uniform measure $\mu(i)=\frac{1}{d},$ for all $i\in E$ play an important role.
In this setting, $W_{\mu}(dr,dz)=\frac{1}{d}\sum_{i=1}^{d}dW_{r}^{i},$ such
that one comes back to a usual diffusion process driven by a
finite-dimensional Brownian motion.
\end{remark}

The following result gives the existence and uniqueness of the solution to our
hybrid system in the class of c\`{a}dl\`{a}g processes in $\mathbb{L}^{1}.$

\begin{theorem}
\label{EU}Suppose that (\ref{Not7}),(\ref{Not6}),(\ref{Not5}),(\ref{Not8}) and
(\ref{G5'}) hold. Then the equation (\ref{Not8a}) has a unique $\mathbb{L}%
^{1}$ solution (that is a $cadlag$ process $X_{s,t}(x),t\geq s$ with
$E(\left\vert X_{s,t}(x)\right\vert )<\infty$ which verifies (\ref{Not8a})$)$.
\end{theorem}

The above theorem has been first proven in \cite{Graham1992}. The main idea is
that, in contrast with the standard approach to $SDEs$ relying on
$\mathbb{L}^{2}$ norms, one has to work here with $\mathbb{L}^{1}$ norms$.$
This is due to the indicator function appearing in the Poisson noise. We
shortly recall this argument in the following.

\subsubsection{Localization Estimates for $\mu$}

For a set $G\subset E$ we denote by $X_{s,t}^{G}$ the solution of the equation
(\ref{Not8a}) in which the measure $\mu$ is restricted to $G$ i.e. substituted
by $1_{G}(z)d\mu(z).$ The first step gives the behavior of such solutions for
different sets $G.$

\begin{lemma}
\label{Distance}We suppose that (\ref{Not7}),(\ref{Not6}),(\ref{Not5}%
),(\ref{Not8}) and (\ref{G5'}) hold. Let $G_{1}\subset G_{2}\subset E$ be two
measurable sets (the case $G_{1}=G_{2}=E$\ is included) and let $\Delta
X_{s,t}=X_{s,t}^{G_{1}}-X_{s,t}^{G_{2}}.$ There exists a universal constant
$C$ such that for every $T\geq0$ one has
\begin{equation}
\mathbb{E}\left[  \sup_{s\leq t\leq T}\left\vert \Delta X_{t}\right\vert
\right]  \leq(\left\vert \Delta X_{s,s}\right\vert +T\alpha(G_{2}\backprime
G_{1}))\exp\left(  CT\left(  \left\Vert \nabla\sigma\right\Vert _{(\mu
,\infty)}+\left\Vert \nabla b\right\Vert _{\infty}+C_{\mu}(\gamma,c)\right)
^{2}+1\right)  . \label{b2}%
\end{equation}

\end{lemma}

The proof is quite straightforward. For our readers' sake, the elements of
proof are gathered in Section \ref{ProofLocalizationLemma}.

Let us now discuss the construction of a solution of the equation
(\ref{Not8a}) and present two alternative representations of this solution.

\subsubsection{Fictive Shocks on Increasing Support Sets}

We fix $G\subset E$ with $\mu(G)<\infty$\ and we recall that $\gamma$ is
upper-bounded by $\Gamma$ (see (\ref{Not5}))$.$ We also fix $s>0$ and we will
construct $X_{s,t}^{G}$ solution of the equation (\ref{Not8a}) associated with
$1_{G}(z)\mu(dz)$ using a compound Poisson process as follows.

One takes $J_{t}$ to be a (usual) Poisson process of parameter $2\Gamma\mu(G)$
and denotes by $T_{k},k\in\mathbb{N}$ the jump times of $J_{t}.$ Moreover, one
considers two sequences of independent random variables $Z_{k}$ and
$U_{k},k\in\mathbb{N}$ ( independent of $J_{t}$ as well and supported by the
set $\Omega$ assumed to be large enough). These random variables are
distributed
\begin{equation}
\mathbb{P}(Z_{k}\in dz)=\frac{1}{\mu(G)}\mu(dz),\quad\mathbb{P}(U_{k}\in
du)=\frac{1}{2\Gamma}1_{[0,2\Gamma]}(u)du. \label{I3}%
\end{equation}
Finally, one defines the continuous stochastic flow $\Phi_{s,t}(x),0\leq s\leq
t$ to be the solution of the $SDE$ (\ref{G5}). Then the solution $X_{s,t}^{G}$
of the equation (\ref{Not8a}) associated with $1_{G}(z)\mu(dz)$ is constructed
by setting $X_{s,s}^{G}(x)=x$ and
\[
X_{s,t}^{G}(x)=\Phi_{T_{k},t}(X_{s,T_{k}}^{G}(x)),\text{ on }T_{k}\leq
t<T_{k+1}%
\]
and
\[
X_{s,T_{k+1}}^{G}(x)=X_{s,T_{k+1}-}^{G}(x)+c(T_{k+1},Z_{k+1},X_{s,T_{k+1}%
-}^{G}(x))1_{\{U_{k}\leq\gamma(T_{k+1},Z_{k+1},X_{s,T_{k+1}-}^{G}(x))\}},
\]
where $X_{s,T_{k+1}-}^{G}(x)=\Phi_{T_{k},T_{k+1}}(X_{s,T_{k}}(x))$. This gives
the solution of the equation
\begin{equation}
\left.
\begin{array}
[c]{c}%
X_{s,t}^{G}=x+\int_{s}^{t}\int_{E}\sigma(r,z,X_{s,r}^{G})W_{\mu}%
(dr,dz)+\int_{s}^{t}b(r,X_{s,r}^{G})dr\\
+\sum_{k=J_{s}+1}^{J_{t}}c(T_{k},Z_{k},X_{s,T_{k}-}^{G}(x))1_{\{U_{k}%
\leq\gamma(T_{k},Z_{k},X_{s,T_{k}-}^{G}(x))\}}%
\end{array}
\right.  \label{I5'}%
\end{equation}
that is (\ref{Not8a}) associated with $1_{G}(z)\mu(dz).$ This is the so-called
"fictive shock" representation (see \cite{LapeyrePardouxSentis1998}).

\begin{remark}
\label{Existance} To construct the global solution, one begins with
considering a sequence $E_{n}\uparrow E$ with $\mu(E_{n})<\infty.$ Then, one
constructs $X_{s,t}^{E_{n}}$ as before and then, using (\ref{b2}), checks that
this is a Cauchy sequence. Passing to the limit, one obtains $X_{s,t}$
solution of the general equation (\ref{Not8a}). Uniqueness follows directly
from (\ref{b2}).
\end{remark}

For the simplicity of the notation, in the following we will work with $s=0.$
The estimates for $s>0$ are quite similar. As usual, we will denote
$X_{0,t}^{G}(x)$ by$~X_{t}^{G}(x).$

\subsubsection{Real Shocks}

We construct now the "real shock" representation $\overline{X}_{t}^{G}$ in the
following way. We define $E_{\ast}=E\cup\{z_{\ast}\},$ where $z_{\ast}$ is a
point which does not belong to $E$ and we extend $\mu$ to $E_{\ast}$ by
setting $\mu(z_{\ast})=1.$ We also extend $c(t,z,x)$ to $E_{\ast}$ by
$c(t,z_{\ast},x)=0,$ for every $(t,x)\in\mathbb{R}_{+}\times\mathbb{R}^{d}.$
Given a sequence $\left(  z_{k}\right)  _{k\in\mathbb{N}}\subset E_{\ast},$
one denotes $z^{k}=(z_{1},...,z_{k})$ and constructs $x_{t}(x,z^{J_{t}})$ as
follows.%
\begin{equation}
\left.
\begin{array}
[c]{l}%
x_{t}(x,z^{k})=\Phi_{T_{k},t}(x_{T_{k}}(x,z^{k})),\quad\text{on }T_{k}\leq
t<T_{k+1},\\
x_{T_{k+1}}(x,z^{k+1})=x_{T_{k+1}-}(x,z^{k})+c(T_{k+1},z_{k+1},x_{T_{k+1}%
-}(x,z^{k}))1_{G}(z_{k+1}).
\end{array}
\right.  \label{I6}%
\end{equation}
Next, we define, for every $(t,z,x)\in%
\mathbb{R}
_{+}\times E_{\ast}\times%
\mathbb{R}
^{d},$
\begin{equation}
\left.
\begin{array}
[c]{l}%
q_{G}(t,z,x)=\Theta_{G}(t,x)1_{\{z_{\ast}\}}(z)+\frac{1}{2\Gamma\mu(G)}%
1_{G}(z)\gamma(t,z,x),\text{ }\\
\text{where }\Theta_{G}(t,x)=1-\frac{1}{2\Gamma\mu(G)}\int_{G}\gamma
(t,z,x)d\mu(z).
\end{array}
\right.  \label{I7}%
\end{equation}
We consider a sequence of random variables $\left(  \overline{Z}_{k}\right)
_{k\in\mathbb{N}}$\ with the laws constructed recursively by
\begin{equation}
E(\overline{Z}_{k}\in dz\mid x_{T_{k}-}(x,\overline{Z}^{k-1})=y)=q_{G}%
(T_{k},z,y)\mu(dz), \label{I8}%
\end{equation}
where $\overline{Z}^{k-1}=(\overline{Z}_{1},...,\overline{Z}_{k-1}).$ Finally,
we define $\overline{X}_{t}^{G}(x)=x_{t}(x,\overline{Z}^{J_{t}}).$ This
amounts to saying that
\begin{equation}
\overline{X}_{t}^{G}(x)=x+\int_{0}^{t}\int_{E}\sigma(s,z,\overline{X}_{s}%
^{G})dW_{\mu}(ds,dz)+\int_{0}^{t}b(s,\overline{X}_{s}^{G})ds+\sum_{k=1}%
^{J_{t}}c(T_{k},\overline{Z}_{k},\overline{X}_{T_{k}-}^{G})1_{G}(\overline
{Z}_{k}). \label{D1}%
\end{equation}
The equation (\ref{D1}) is similar to the equation (\ref{I5'}) but now
$1_{\{U_{k}\leq\gamma(T_{k},Z_{k},X_{T_{k}-}(x))\}}$ no longer appears.

\begin{remark}
If $\sigma$ and $b$ are smooth functions, one can choose a variant of
$x\mapsto\Phi_{s,t}(x)$ that is almost surely differentiable. Moreover, if
$x\mapsto c(t,z,x)$ is also smooth, then $x\mapsto x_{t}(x,z^{J_{t}})$ is
smooth as well. So $x\mapsto\overline{X}_{t}^{G}(x)$ will be also
differentiable. This is less obvious for $x\mapsto X_{t}^{G}(x)$ because of
the indicator function appearing in the equation (\ref{I5'}).
\end{remark}

Moreover, we have the following well known identity of laws result.

\begin{lemma}
\label{Law}The law of $\left(  \overline{X}_{t}^{G}(x)\right)  _{t\geq0}$
coincides with the law of $\left(  X_{t}^{G}(x)\right)  _{t\geq0}$ solution to
(\ref{I5'}). In particular, $\mathcal{P}_{t}^{G}f(x):=\mathbb{E}\left[
f(X_{t}^{G}(x)\right]  =\mathbb{E}\left[  f(\overline{X}_{t}^{G}(x)\right]  .$
\end{lemma}

\section{Differentiability of the Flow\label{SectDifferentiability}}

We will now study the differentiability of the application $x\mapsto
\overline{X}_{t}^{G}(x)$ when assuming $\mu(G)<\infty.$ Let us begin with
introducing some further notations. Given a regular function $g:\mathbb{R}%
_{+}\times E\times\mathbb{R}^{d}\rightarrow\mathbb{R}$ that is differentiable
with respect to the space variable $x\in\mathbb{R}^{d}$ we denote by%
\begin{align}
&  \left\vert g\right\vert _{G,p}=\sup_{0\leq t\leq T}\sup_{x\in\mathbb{R}%
^{d}}\left(  \int_{G}\left\vert g(t,z,x)\right\vert ^{p}\gamma(t,z,x)\mu
(dz)\right)  ^{\frac{1}{p}},\label{D5}\\
&  [g]_{G,p}=\sup_{1\leq p^{\prime}\leq p}\left\vert g\right\vert
_{G,p^{\prime}}.\label{D5b}\\
&  \theta_{q,p}(G)=1+\left\Vert \sigma\right\Vert _{2,q,(\mu,\infty
)}+\left\Vert b\right\Vert _{2,q,\infty}+\sum_{2\leq\left\vert \alpha
\right\vert \leq q}[\partial_{x}^{\alpha}c]_{G,p},\label{D5'}\\
&  a_{p}(G)=\left\Vert \nabla\sigma\right\Vert _{(\mu,\infty)}^{2}+\left\Vert
\nabla b\right\Vert _{\infty}+[\nabla c]_{G,p}^{p},\label{D5a}\\
&  \alpha_{q,p}(C,G)=C\theta_{q,pq}^{q%
{\textstyle\sum\limits_{1\leq n\leq q}}
\frac{1}{n}}(G)\exp\left(  CTq%
{\textstyle\sum\limits_{1\leq n\leq q}}
\frac{1}{n}a_{pq}(G)\right)  . \label{D5''}%
\end{align}
Note that if $q_{1}\leq q_{2}$ and $p_{1}\leq p_{2}$ then $\theta_{q_{1}%
,p_{1}}(G)\leq\theta_{q_{2},p_{2}}(G)$ and $a_{p_{1}}(G)\leq a_{p_{2}}(G)$
(this is the reason of being of $\sup_{1\leq p^{\prime}\leq p}$ in
$[g]_{G,p}).$

In the following we suppose that $\mu(G)<\infty$ and $\theta_{q,p}(G)<\infty
$\ and consider $X_{t}^{G}(x)$ and $\overline{X}_{t}^{G}(x),$ solutions of the
equations (\ref{I5'}) and (\ref{D1}) constructed in the previous section.
Under these hypothesis one may choose a variant of $x\mapsto\overline{X}%
_{t}^{G}(x)$ which is $q$ times differentiable. Our aim is to estimate the
$\mathbb{L}^{p}$ norm of $\partial^{\alpha}\overline{X}_{t}^{G}(x):$

\begin{lemma}
\label{EstimatesPD}Let $\alpha$ be a multi-index with $\left\vert
\alpha\right\vert =k.$\ For every $p\geq2,$ the following inequality holds
true.%
\begin{equation}
\sup_{x\in\mathbb{R}^{d}}\mathbb{E}\left[  \sup_{t\leq T}\left\vert
\partial^{\alpha}\overline{X}_{t}^{G}(x)\right\vert ^{p}\right]  ^{\frac{1}%
{p}}\leq\alpha_{k,p}(C,G). \label{D6}%
\end{equation}

\end{lemma}

The proof is postponed to Section \ref{SectionProofLemmaEstimates}. The main
idea consists in providing estimates for the chain rule distinguishing first
order and higher order derivatives. Subsequently, these estimates will be
applied for the different components in the differential formula of
$\partial^{\alpha}\overline{X}_{t}^{G}(x).$ Next, one provides estimates for
generic equations of this type (having a linear form) and uses a recurrence
argument over $\left\vert \alpha\right\vert .$

We give now some consequences of (\ref{D6}).

\begin{corollary}
\label{CorollaryDiff1}\textbf{A}. Let $\alpha$ be a multi-index with
$\left\vert \alpha\right\vert =q\geq1$ and let $p\geq2$ and $\eta>0$ be
given$.$ For every $g:\mathbb{R}_{+}\times\mathbb{R}^{d}\rightarrow\mathbb{R}$
that is smooth with respect to $x\in\mathbb{R}^{d},$ the following inequality
holds true.
\begin{equation}
\left\Vert \sum_{k=1}^{J_{t}}\left\vert \partial^{\alpha}g\left(
T_{k},\overline{X}_{T_{k}-}^{G}(x)\right)  \right\vert \right\Vert _{p}\leq
C\left\Vert g\right\Vert _{1,q,\infty}\Gamma\mu(G)\left(  t\vee1\right)
\alpha_{q,(1+\eta)pq}^{q}(C,G). \label{D12}%
\end{equation}

\textbf{B} Let $g:\mathbb{R}_{+}\times\mathbb{R}^{d}\rightarrow\mathbb{R}$ be
a function that is smooth with respect to $x\in\mathbb{R}^{d}.$ Then,
\begin{equation}
\left\Vert \sum_{k=1}^{J_{t}}1_{G}(\overline{Z}_{k})\left\vert \partial
^{\alpha}g\left(  T_{k},\overline{Z}_{k},\overline{X}_{T_{k}-}^{G}(x)\right)
\right\vert \right\Vert _{p}\leq C\left(  \sum_{1\leq\left\vert \beta
\right\vert \leq\left\vert \alpha\right\vert }[\partial^{\beta}g]_{G,(1+\eta
)p}\alpha_{q,\frac{(1+\eta)}{\eta}pq}^{q}(C,G)\right)  . \label{D13}%
\end{equation}

\end{corollary}

For our readers' sake, the proof is provided in Section
\ref{SectionProofCorDiff}.

\section{Differentiability of the Semigroup\label{SectionDiffSemigroup}}

Before giving the main result on the semigroup of operators, we recall that
the law of $\overline{Z}^{Jt}=(\overline{Z}_{1},...,\overline{Z}_{J_{t}})$
has, as density, $p_{J_{t}}(x,z^{J_{t}})\mu(dz_{1}),...,\mu(dz_{J_{t}})$ with
\[
p_{J_{t}}(x,z^{J_{t}})=%
{\textstyle\prod_{k=1}^{J_{t}}}
q_{G}(T_{k},z_{k},x_{T_{k}-}(x,z^{k-1})).
\]
This explicit formulation allows one to obtain a very important step in
proving regularity of the semigroup.

\begin{lemma}
\label{CorollaryDiff2}Let us assume that $\Gamma\mu(G)\geq1$ and
$\alpha_{q,2qp}(C,G)<\infty$ for some given $q\in\mathbb{N}\ $and $p\geq2.$
Let $\alpha$ be a multi-index with $\left\vert \alpha\right\vert =q\geq1.$
Then, there exists a universal constant $C$ (depending on $p$ and $q$ but not
on $G)$ such that%
\begin{equation}
\left\Vert \partial^{\alpha}\ln p_{J_{t}}(x,\overline{Z}^{J_{t}})\right\Vert
_{p}\leq C\left(  t\vee1\right)  \times\alpha_{q,2pq}^{q}(C,G))\times
(\Gamma_{G,q}(\gamma)+%
{\textstyle\sum_{1\leq\left\vert \beta\right\vert \leq q}}
[\partial^{\beta}\ln\gamma]_{G,2p}) \label{D14}%
\end{equation}
with $[\ln\gamma]_{G,2p}$\ defined in (\ref{D5b}) and
\begin{equation}
\Gamma_{G,q}(\gamma)=\sup_{t\leq T}\sup_{x\in R^{d}}%
{\textstyle\sum_{h=1}^{q}}
{\textstyle\sum_{1\leq\left\vert \rho\right\vert \leq h}}
\left(  \int_{G}\left\vert \partial^{\rho}\ln\gamma(t,z,x)\right\vert
^{\frac{h}{\left\vert \rho\right\vert }}\gamma(t,z,x)\mu(dz)\right)
^{\frac{q}{h}}. \label{D14'}%
\end{equation}

\end{lemma}

Before going any further, we make the following elementary remark.

\begin{remark}
For every smooth function $\phi:\mathbb{R}^{d}\rightarrow\mathbb{R}_{+}^{\ast
}$ and any multi-index $\rho$ with $\left\vert \rho\right\vert =q,$ one gets
the existence of some function $P_{\rho}^{\phi}:\mathbb{R}^{d}\rightarrow
\mathbb{R}_{+}^{\ast}$ such that%
\begin{equation}
\partial^{\rho}\phi(x)=\phi(x)P_{\rho}^{\phi}(x)\text{ and }\left\vert
P_{\rho}^{\phi}(x)\right\vert \leq C%
{\textstyle\sum_{1\leq\left\vert \beta\right\vert \leq q}}
\left\vert \partial^{\beta}\ln\phi(x)\right\vert ^{\frac{q}{\left\vert
\beta\right\vert }}, \label{D15}%
\end{equation}
for all $x\in\mathbb{R}^{d}.$ In order to prove this one first writes
$\phi=\exp(\ln\phi))$ and then takes derivatives. One obtains $\phi$
multiplied with a polynomial applied to terms of type $\partial^{\beta}\ln
\phi$ i.e. some linear combination of products of type $\prod_{i=1}%
^{r}\partial^{\beta(i)}\ln\phi$ with $\sum_{i=1}^{r}\left\vert \beta
(i)\right\vert =q.$ Using Young's inequality with $p_{i}=\frac{q}{\left\vert
\beta(i)\right\vert },$ we obtain%
\[
\left\vert
{\textstyle\prod_{i=1}^{r}}
\partial^{\beta(i)}\ln\phi\right\vert \leq%
{\textstyle\sum_{i=1}^{r}}
\frac{\left\vert \beta(i)\right\vert }{q}\left\vert \partial^{\beta(i)}\ln
\phi\right\vert ^{\frac{q}{\left\vert \beta(i)\right\vert }}.
\]
And this proves the upper bound for $\left\vert P_{\rho}^{\phi}(x)\right\vert
$ given in (\ref{D15}).
\end{remark}

We are now able to proceed with the proof of Lemma \ref{CorollaryDiff2}.

\begin{proof}
We have%
\[
\left.
\begin{array}
[c]{l}%
\partial^{\alpha}\ln p_{J_{t}}(x,z^{J_{t}})=\overset{J_{t}}{\underset{k=1}{%
{\textstyle\sum}
}}1_{\{z_{\ast}\}}(z_{k})\partial^{\alpha}(\ln\Theta_{G}(T_{k},x_{T_{k}%
-}(x,z^{k-1})))\\
\text{ \ \ \ \ \ \ \ \ \ \ \ \ \ \ \ \ \ \ \ \ \ \ \ \ }+\overset{J_{t}%
}{\underset{k=1}{%
{\textstyle\sum}
}}1_{G}(z_{k})\partial^{\alpha}(\ln\gamma(T_{k},z_{k},x_{T_{k}-}%
(x,z^{k-1})))\\
\text{ \ \ \ \ \ \ \ \ \ \ \ \ \ \ \ \ \ \ \ }=:s_{1}(x,z^{J_{t}}%
)+s_{2}(x,z^{J_{t}}).
\end{array}
\right.
\]
In order to estimate $s_{1}(x,z^{J_{t}})$ we will use (\ref{D12}) for
$g=\ln\Theta_{G}.$ Recalling that $\gamma$ is upper-bounded by $\Gamma,$ we
have $\Theta_{G}(t,x)\geq\frac{1}{2}.$ Then, for every multi-index $\alpha$
with $\left\vert \alpha\right\vert =q$ one has
\[
\left.
\begin{array}
[c]{l}%
\left\vert \partial^{\alpha}\ln\Theta_{G}(t,x)\right\vert \leq%
{\textstyle\sum_{r=1}^{q}}
{\textstyle\sum_{\left\vert \beta(1)\right\vert +...+\left\vert \beta
(r)\right\vert =q}}
{\textstyle\prod_{i=1}^{r}}
\left\vert \partial^{\beta(i)}\Theta_{G}(t,x)\right\vert \\
\text{ \ \ \ \ \ \ \ \ \ \ \ \ \ \ \ \ \ \ }\leq%
{\textstyle\sum_{r=1}^{q}}
\frac{C}{(\Gamma\mu(G))^{r}}%
{\textstyle\sum_{\left\vert \beta(1)\right\vert +...+\left\vert \beta
(r)\right\vert =q}}
{\textstyle\prod_{i=1}^{r}}
\int_{G}\left\vert \partial^{\beta(i)}\gamma(t,z,x)\right\vert \mu(dz).
\end{array}
\right.
\]
Using Young's inequality and (\ref{D15}) (recall that $\Gamma\mu(G)\geq1),$
one proves%
\[
\left.
\begin{array}
[c]{l}%
\overset{r}{\underset{i=1}{%
{\textstyle\prod}
}}\int_{G}\left\vert \partial^{\beta(i)}\gamma(t,z,x)\right\vert \mu(dz)\leq
C\underset{i=1}{\overset{r}{%
{\textstyle\sum}
}}\left(  \int_{G}\left\vert \partial^{\beta(i)}\gamma(t,z,x)\right\vert
d\mu(z)\right)  ^{\frac{q}{\left\vert \beta(i)\right\vert }}\\
\text{ \ \ \ \ \ \ \ \ \ \ \ \ \ \ \ \ \ \ \ \ \ \ \ \ \ \ \ \ \ \ \ \ \ \ }%
=C\underset{i=1}{\overset{r}{%
{\textstyle\sum}
}}\left(  \int_{G}\left\vert P_{\beta(i)}^{\gamma}(t,z,x)\right\vert
\gamma(t,z,x)\mu(dz)\right)  ^{\frac{q}{\left\vert \beta(i)\right\vert }}\\
\text{ \ \ \ \ \ \ \ \ \ \ \ \ \ \ \ \ \ \ \ \ \ \ \ \ \ \ \ \ \ \ \ \ \ \ }%
\leq C\underset{1\leq\left\vert \rho\right\vert \leq h\leq q}{%
{\textstyle\sum}
}\left(  \int_{G}\left\vert \partial^{\rho}\ln\gamma(t,z,x)\right\vert
^{\frac{h}{\left\vert \rho\right\vert }}\gamma(t,z,x)\mu(dz)\right)
^{\frac{q}{h}}\leq C\Gamma_{G,q}(\gamma).
\end{array}
\right.
\]
We conclude that $\left\vert \partial^{\alpha}\ln\Theta_{G}(t,x)\right\vert
\leq\frac{C}{\Gamma\mu(G)}\Gamma_{G,q}(\gamma).$ As a consequence of
(\ref{D12}) (with $\eta=1$), one gets%
\[
\left(  \mathbb{E}\left[  \left\vert s_{1}\left(  x,\overline{Z}^{J_{t}%
}\right)  \right\vert ^{p}\right]  \right)  ^{\frac{1}{p}}\leq C\left(
t\vee1\right)  \Gamma_{G,q}(\gamma)\alpha_{q,2pq}^{q}(C,G).
\]
To estimate the second term, we use (\ref{D13}) with $g(t,z,x)=\ln
\gamma(t,z,x)$ and for $\eta=1$ to get an upper bound given by
\[
\mathbb{E}\left[  \left\vert
{\textstyle\sum_{k=1}^{J_{t}}}
1_{G}(\overline{Z}_{k})\partial^{\alpha}\ln\gamma\left(  T_{k},z_{k},\left(
\overline{X}_{T_{k}-}(x)\right)  \right)  \right\vert ^{p}\right]  \leq
C\alpha_{q,2pq}^{q}(C,G)%
{\textstyle\sum_{1\leq\left\vert \beta\right\vert \leq q}}
[\partial^{\beta}\ln\gamma]_{G,2p}.
\]
The proof is now complete.
\end{proof}

We now discuss the differentiability of the semigroup associated with our
process. To this purpose, we let, for $f$ regular enough,%
\[
\mathcal{P}_{t}^{G}f(x)=\mathbb{E}\left[  f\left(  X_{t}^{G}(x)\right)
\right]  =\mathbb{E}\left[  f\left(  \overline{X}_{t}^{G}(x)\right)  \right]
=\mathbb{E}\left[  \int_{E^{J_{t}}}f(x_{t}(x,z^{J_{t}}))p_{J_{t}}(x,z^{J_{t}%
})\mu(dz_{1})...\mu(dz_{J_{t}})\right]  .
\]

\begin{theorem}
\label{ThDiffSemigroup1}We assume (\ref{Not7}),(\ref{Not6}),(\ref{Not5}%
),(\ref{Not8}) and (\ref{G5'}) to hold true. Then, for every $q\in\mathbb{N},$
there exists a constant $C>0$ independent of $G$ such that%
\begin{equation}
\left\Vert \mathcal{P}_{t}^{G}f\right\Vert _{q,\infty}\leq C\left\Vert
f\right\Vert _{q,\infty}\times\left(  t\vee1\right)  ^{q}\times\alpha
_{q,4q}^{2q}(C,G)\times\left(  1+\Gamma_{G,q}(\gamma)+\sum_{1\leq\left\vert
\beta\right\vert \leq q}[\partial^{\beta}\ln\gamma]_{G,4q}\right)  ^{q}.
\label{D16}%
\end{equation}

\end{theorem}

\begin{proof}
We want to estimate, for a multi-index $\alpha$ such that $\left\vert
\alpha\right\vert \leq q,$ the partial derivative%
\begin{align*}
\partial^{\alpha}\mathbb{E}\left[  f\left(  \overline{X}_{t}^{G}(x)\right)
\right]   &  =\sum_{(\beta,\rho)=\alpha}\mathbb{E}\left[  \int_{E^{J_{t}}%
}\partial^{\beta}f\left(  x_{t}\left(  x,z^{J_{t}}\right)  \right)
\times\partial^{\rho}p_{J_{t}}(x,z^{J_{t}})\mu(dz_{1})...\mu(dz_{J_{t}%
})\right] \\
&  =\sum_{(\beta,\rho)=\alpha}\mathbb{E}\left[  \int_{E^{J_{t}}}%
\partial^{\beta}f\left(  x_{t}\left(  x,z^{J_{t}}\right)  \right)  P_{\rho
}^{p_{J_{t}}(\cdot,z^{J_{t}})}(x,z^{J_{t}})p_{J_{t}}(x,z^{J_{t}})\mu
(dz_{1})...\mu(dz_{J_{t}})\right]
\end{align*}
with $P_{\rho}^{p_{J_{t}}(\cdot,z^{J_{t}})}(x,z^{J_{t}})$\ given by
(\ref{D15}).\ Using Cauchy-Schwartz inequality, we have%
\[
\left.
\begin{array}
[c]{l}%
\left\vert \partial^{\alpha}\mathbb{E}\left[  f\left(  \overline{X}_{t}%
^{G}(x)\right)  \right]  \right\vert \leq\sum_{(\beta,\rho)=\alpha}A_{\beta
}^{\frac{1}{2}}\times B_{\rho}^{\frac{1}{2}},\text{ with }A_{\beta}%
=\mathbb{E}\left[  \left\vert \partial^{\beta}f\left(  \overline{X}_{t}%
^{G}(x)\right)  \right\vert ^{2}\right]  \text{ and}\\
B_{\rho}=\mathbb{E}\left[  \int_{E^{J_{t}}}\left\vert P_{\rho}^{p_{J_{t}%
}(\cdot,z^{J_{t}})}(x,z^{J_{t}})\right\vert ^{2}\times p_{J_{t}}(x,z^{J_{t}%
})\mu(dz_{1})...\mu(dz_{J_{t}})\right]  .
\end{array}
\right.
\]
We have (recalling that $0\leq\left\vert \beta\right\vert \leq\left\vert
\alpha\right\vert \leq q$),%
\[
\left\vert A_{\beta}\right\vert \leq C\left\Vert f\right\Vert _{q,\infty}%
^{2}\left(  1+\sum_{1\leq\left\vert \rho\right\vert \leq q}\mathbb{E}\left[
\left\vert \partial^{\rho}\overline{X}_{t}(x)\right\vert ^{2q}\right]
\right)  \leq C\left\Vert f\right\Vert _{q,\infty}^{2}\alpha_{q,2q}%
^{2q}(C,G).
\]
If $\left\vert \rho\right\vert =0$ then $B_{\rho}=1.$ Moreover, using the
estimates from (\ref{D15}) and (\ref{D14}) respectively (with $p=\frac
{2\left\vert \rho\right\vert }{\left\vert \beta\right\vert }\geq2),$ one gets%
\begin{align*}
\sum_{1\leq\left\vert \rho\right\vert \leq q}B_{\rho}^{\frac{1}{2}}  &  \leq
C\sum_{1\leq\left\vert \rho\right\vert \leq q}\sum_{1\leq\left\vert
\beta\right\vert \leq\left\vert \rho\right\vert }\left(  \mathbb{E}\left[
\int_{E^{J_{t}}}\left\vert \partial^{\beta}\ln p_{J_{t}}\left(  x,z^{J_{t}%
}\right)  \right\vert ^{\frac{2\left\vert \rho\right\vert }{\left\vert
\beta\right\vert }}\times p_{J_{t}}(x,z^{J_{t}})d\mu(z_{1})...d\mu(z_{J_{t}%
})\right]  \right)  ^{\frac{1}{2}}\\
&  \leq C\sum_{1\leq\left\vert \beta\right\vert \leq\left\vert \rho\right\vert
\leq q}\left(  \mathbb{E}\left[  \left\vert \partial^{\beta}\ln p_{J_{t}%
}\left(  x,\overline{Z}^{J_{t}}\right)  \right\vert ^{\frac{2\left\vert
\rho\right\vert }{\left\vert \beta\right\vert }}\right]  \right)  ^{\frac
{1}{2}}\\
&  \leq C\times\left(  t\vee1\right)  ^{q}\times\alpha_{q,4q}^{q}(C,G)\left(
1+\Gamma_{G,q}(\gamma)+\sum_{1\leq\left\vert \beta\right\vert \leq q}%
[\partial^{\beta}\ln\gamma]_{G,4q}\right)  ^{q}%
\end{align*}
The assertion follows from these estimates.
\end{proof}

In the proof of the previous theorem we need $\mu(G)<\infty$ having to argue
on $X_{t}^{G}$ and $\overline{X}_{t}^{G}.$ We take an increasing sequence
$E_{n}\uparrow E$ such that $\mu(E_{n})<\infty$ and we use (\ref{D16}) and
(\ref{b2}), to extend the result to (possibly) infinite total measure
$\mu\left(  E\right)  .$ For simplicity, we will write $\mathcal{P}_{t}$
instead of $\mathcal{P}_{t}^{E}.$

\begin{theorem}
\label{M}We assume that the jump rate $\gamma$ is bounded (\ref{Not5}), the
jump coefficients are Lipschitz regular (\ref{Not6}) resp. the diffusion
coefficients are smooth (\ref{G5'}). Moreover, we assume the integrability
conditions on the jump mechanism (\ref{Not7}) and (\ref{Not8}).

Then $\mathcal{P}_{t}$ maps $C_{b}^{q}(\mathbb{R}^{d})$ in $C_{b}%
^{q}(\mathbb{R}^{d})$ and there exists $C>0$ (independent of $E$) such that
\begin{equation}
\left\Vert \mathcal{P}_{t}f\right\Vert _{q,\infty}\leq C\left\Vert
f\right\Vert _{q,\infty}(t\vee1)^{q}\times\alpha_{q,4q}^{2q}(C,E)\times\left(
1+\Gamma_{E,q}(\gamma)+\sum_{1\leq\left\vert \beta\right\vert \leq q}%
[\partial^{\beta}\ln\gamma]_{E,4q}\right)  ^{q}. \label{D18}%
\end{equation}
with $\alpha_{q,p}(C,E),$ $\Gamma_{E,q}(\gamma)$ and $[\partial^{\beta}%
\ln\gamma]_{E,p}$ defined in (\ref{D5''}),(\ref{D14'}) and (\ref{D5b}).
\end{theorem}

\section{The Distance Between Two Semigroups\label{SectionDistanceSemigroups}}

In this section we consider two sets of coefficients $\sigma,b,c,\gamma$ and
$\widehat{\sigma},\widehat{b},\widehat{c},\widehat{\gamma}$ on measurable
space $(E,\mathcal{E},\mu)$ respectively $(\widehat{E},\widehat{\mathcal{E}%
},\widehat{\mu})$ and we associate the stochastic equations in (\ref{Not8a}).
The space $\left(  \Omega,\mathcal{F},\mathbb{P}\right)  $ is assumed to be
large enough to support the (possibly mutually independent) Poisson random
measures $N_{\mu}$ and $N_{\widehat{\mu}}$ as well as the cylindrical Brownian
processes $W_{\mu}$ and $W_{\widehat{\mu}}.$

We denote by $X_{t_{0},t}\left(  x\right)  $ respectively by $\widehat
{X}_{t_{0},t}(x)$ the solutions of the corresponding equations and we consider
the non homogeneous semigroups $\mathcal{P}_{t_{0},t}f(x)=\mathbb{E}\left[
f\left(  X_{t_{0},t}(x)\right)  \right]  $ and $\widehat{\mathcal{P}}%
_{t_{0},t}f(x)=\mathbb{E}\left[  f\left(  \widehat{X}_{t_{0},t}(x)\right)
\right]  .$ Our aim is to estimate the weak distance between these two
semigroups. To begin, we give the standing assumptions.

\subsection{Standing Assumptions}

\begin{assumption}
[Assumption $H_{1}(q)$]Given the coefficients $\sigma,b,c,\gamma,$\ we denote
by
\begin{equation}
Q_{q}(T,\mathcal{P}):=C(T\vee1)^{q}\times\alpha_{q,4q}^{2q}(C,E))\times\left(
1+\Gamma_{E,q}(\gamma)+\sum_{1\leq\left\vert \beta\right\vert \leq q}%
[\partial^{\beta}\ln\gamma]_{E,4q})\right)  ^{q}, \label{Dist1}%
\end{equation}
with $\alpha_{q,p}(C,E),$ $\Gamma_{E,q}(\gamma)$ and $[\partial^{\beta}%
\ln\gamma]_{E,p}$ defined in (\ref{D5''}),(\ref{D14'}) and (\ref{D5b}) and the
constant $C$ appearing in (\ref{D16}). We assume that all these quantities are
well defined and finite, so that $Q_{q}(T,\mathcal{P})<\infty.$
\end{assumption}

Whenever this assumption holds true, Theorem \ref{M} yields%
\begin{equation}
\sup_{t_{0}\leq t\leq T}\left\Vert \mathcal{P}_{t_{0},t}f\right\Vert
_{q,\infty}\leq Q_{q}(T,\mathcal{P})\left\Vert f\right\Vert _{q,\infty}
\label{D19}%
\end{equation}

We will also need a condition on the behavior of particular (polynomial) test functions.

\begin{assumption}
[Assumption\textbf{ }$H_{2}\left(  k\right)  $]For $k\in\mathbb{N}$, we denote
by $\psi_{k}(x)=(1+\left\vert x\right\vert ^{2})^{\frac{k}{2}}$ and we assume
that one finds a constant $C_{k}(T,\mathcal{P})$ such that%
\begin{equation}
\sup_{t_{0}\leq t\leq T}\left\Vert \frac{1}{\psi_{k}}\mathcal{P}_{t_{0},t}%
\psi_{k}\right\Vert _{\infty}\leq C_{k}(T,\mathcal{P})<\infty. \label{Dist2}%
\end{equation}

\end{assumption}

Finally, we will make an assumption on the gradient of the infinitesimal
operators
\begin{equation}
\mathcal{L}_{t}f(x)=\frac{1}{2}Tr\left[  a(t,x)\partial^{2}f(x)\right]
+b(t,x)\partial f(x)+\int_{E}(f(x+c(t,z,x))-f(x))\gamma(t,z,x)\mu(dz),
\label{Dist3}%
\end{equation}
with
\[
a^{i,j}(t,x)=\int_{E}\sigma^{i}(t,z,x)\sigma^{j}(t,z,x)\mu(dz),\text{ for all
}x\in%
\mathbb{R}
^{d},\text{ }t\in\left[  0,T\right]  \text{ and all }1\leq i,j\leq d.
\]

\begin{assumption}
[Assumption $H_{3}(k,q)$]We assume that there exists $C\geq1$ such that, for
all $f\in C_{b}^{q}\left(  \mathbb{R}^{d}\right)  $%
\begin{equation}
\sup_{t\leq T}\left\Vert \frac{1}{\psi_{k}}\nabla\mathcal{L}_{t}f\right\Vert
_{\infty}\leq C\left\Vert f\right\Vert _{q,\infty}. \label{Dist3a}%
\end{equation}

\end{assumption}

When $\sigma=0,$ one can ask this condition for $q\geq2.$ Otherwise, one
usually takes $q\geq3.$

\subsection{Upper Bounds on the Distance Between Semigroups}

We consider two sets of coefficients $\sigma,b,c,\gamma$ and $\widehat{\sigma
},\widehat{b},\widehat{c},\widehat{\gamma}$ and the corresponding semigroups
$\mathcal{P}_{t}$ and $\widehat{\mathcal{P}}_{t}.$ We fix $k,q\in\mathbb{N}$.

\begin{theorem}
\label{Conv} We assume that $\mathcal{P}_{t}$ satisfies $H_{2}(k)$ and
$H_{3}(k,q)$ and that $\widehat{\mathcal{P}}_{t}$ verifies $H_{1}(q)$ and
$H_{3}(k,q).$ Moreover, we assume (\ref{Not5}),(\ref{Not6}),(\ref{Not7}) and
(\ref{G5'}) to hold true for $\widehat{\sigma},\widehat{b},\widehat
{c},\widehat{\gamma}.$ Finally, we assume that there exists a function
$\varepsilon(\cdot):%
\mathbb{R}
_{+}\longrightarrow%
\mathbb{R}
_{+}$ such that, for every $0\leq t\leq T$
\begin{equation}
\left\Vert \frac{1}{\psi_{k}}(\mathcal{L}_{t}-\widehat{\mathcal{L}}%
_{t})f\right\Vert _{\infty}\leq\varepsilon(t)\left\Vert f\right\Vert
_{q,\infty}. \label{dist4'}%
\end{equation}
Then, the following inequality holds true%
\begin{equation}
\left\Vert \frac{1}{\psi_{k}}(\mathcal{P}_{t_{0},t}-\widehat{\mathcal{P}%
}_{t_{0},t})f\right\Vert _{\infty}\leq C_{k}(T,\mathcal{P})Q_{q}%
(T,\widehat{\mathcal{P}})\left\Vert f\right\Vert _{q,\infty}\times\int_{t_{0}%
}^{t}\varepsilon(s)ds, \label{dist6}%
\end{equation}
with $C_{k}(\mathcal{P})$ the constant in (\ref{Dist2}) (with respect to
$\sigma,b,c,\gamma)$ and $Q_{q}(T,\widehat{\mathcal{P}})$ is given in
(\ref{Dist1}) (with respect to the coefficients $\widehat{\sigma},\widehat
{b},\widehat{c},\widehat{\gamma}).$
\end{theorem}

\begin{proof}
For $n\in\mathbb{N}$, we set $\delta:=\frac{t-t_{0}}{n}$ and $t_{i}%
=t_{0}+i\delta,$ for all $0\leq i\leq n.$ With these notations,%
\[
\frac{1}{\psi_{k}}(\mathcal{P}_{t_{0},t}-\widehat{\mathcal{P}}_{t_{0}%
,t})f=\sum_{i=0}^{n-1}\frac{1}{\psi_{k}}\mathcal{P}_{t_{i+1},t}\psi_{k}%
\frac{1}{\psi_{k}}(\mathcal{P}_{t_{i},t_{i+1}}-\widehat{\mathcal{P}}%
_{t_{i},t_{i+1}})\widehat{\mathcal{P}}_{t_{0},t_{i}}f.
\]
We make the notation(s) $g_{i}=\widehat{\mathcal{P}}_{t_{0},t_{i}}f.$ Using
(\ref{Dist2}) for $\mathcal{P}_{t_{i+1},t}$
\[
\left\Vert \frac{1}{\psi_{k}}(\mathcal{P}_{t_{0},t}-\widehat{\mathcal{P}%
}_{t_{0},t})f\right\Vert _{\infty}\leq C_{k}(T,\mathcal{P})\sum_{i=0}%
^{n-1}\left\Vert \frac{1}{\psi_{k}}(\mathcal{P}_{t_{i},t_{i+1}}-\widehat
{\mathcal{P}}_{t_{i},t_{i+1}})g_{i}\right\Vert _{\infty}.
\]
By It\^{o}'s formula,%
\[
\mathcal{P}_{t_{i},t_{i+1}}g_{i}\left(  x\right)  =g_{i}\left(  x\right)
+\mathbb{E}\left[  \int_{t_{i}}^{t_{i+1}}\mathcal{L}_{s}g_{i}(X_{t_{i}%
,s}\left(  x\right)  )ds\right]  =g_{i}\left(  x\right)  +\int_{t_{i}%
}^{t_{i+1}}\mathcal{L}_{s}g_{i}(x)ds+\varepsilon_{i},
\]
with $\varepsilon_{i}(x):=\mathbb{E}\left[  \int_{t_{i}}^{t_{i+1}}\left(
\mathcal{L}_{s}g_{i}(X_{t_{i},s}\left(  x\right)  )-\mathcal{L}_{s}%
g_{i}(x)\right)  ds\right]  .$ We write the same type of formulae for
$\widehat{\mathcal{P}}_{t_{i},t_{i+1}}g_{i},$ take the difference between the
two and use (\ref{dist4'}) in order to get%
\[
\left.
\begin{array}
[c]{c}%
\left\Vert \frac{1}{\psi_{k}}(\mathcal{P}_{t_{i},t_{i+1}}-\widehat
{\mathcal{P}}_{t_{i},t_{i+1}})g_{i}\right\Vert _{\infty}\leq\int_{t_{i}%
}^{t_{i+1}}\left\Vert \frac{1}{\psi_{k}}(\mathcal{L}_{s}-\widehat{\mathcal{L}%
}_{s})g_{i}\right\Vert _{\infty}ds+\left\Vert \frac{1}{\psi_{k}}%
\varepsilon_{i}\right\Vert _{\infty}+\left\Vert \frac{1}{\psi_{k}}%
\widehat{\varepsilon}_{i}\right\Vert _{\infty}\\
\text{ \ \ \ \ \ \ \ \ \ \ \ \ \ \ \ \ \ \ \ \ \ \ \ \ \ \ \ }\leq\left\Vert
g_{i}\right\Vert _{q,\infty}\int_{t_{i}}^{t_{i+1}}\varepsilon(s)ds+\left\Vert
\frac{1}{\psi_{k}}\varepsilon_{i}\right\Vert _{\infty}+\left\Vert \frac
{1}{\psi_{k}}\widehat{\varepsilon}_{i}\right\Vert _{\infty}.
\end{array}
\right.
\]
By (\ref{D19}), $\left\Vert g_{i}\right\Vert _{q,\infty}\leq Q_{q}%
(T,\widehat{\mathcal{P}})\left\Vert f\right\Vert _{q,\infty}$ so that,
finally,%
\[
\left\Vert \frac{1}{\psi_{k}}(\mathcal{P}_{t_{0},t}-\widehat{\mathcal{P}%
}_{t_{0},t})f\right\Vert _{\infty}\leq C_{k}(T,\mathcal{P})\left[
Q_{q}(T,\widehat{\mathcal{P}})\left\Vert f\right\Vert _{q,\infty}\int_{t_{0}%
}^{t}\varepsilon(s)ds+\sum_{i=0}^{n-1}\left(  \left\Vert \frac{1}{\psi_{k}%
}\varepsilon_{i}\right\Vert _{\infty}+\left\Vert \frac{1}{\psi_{k}}%
\widehat{\varepsilon}_{i}\right\Vert _{\infty}\right)  \right]  .
\]
To conclude, one still needs to estimate the terms $\varepsilon_{i}$ and prove
that these errors vanish as $n$ increases. The assumption (\ref{Dist3a})
yields
\[
\left.
\begin{array}
[c]{c}%
\left\vert \mathcal{L}_{s}g_{i}(X_{t_{i},s}(x))-\mathcal{L}_{s}g_{i}%
(x)\right\vert \leq\int_{0}^{1}\left\vert \left\langle \nabla\mathcal{L}%
_{s}g_{i}(\lambda x+(1-\lambda)X_{t_{i},s}(x)),X_{t_{i},s}(x)-x\right\rangle
\right\vert d\lambda\\
\text{
\ \ \ \ \ \ \ \ \ \ \ \ \ \ \ \ \ \ \ \ \ \ \ \ \ \ \ \ \ \ \ \ \ \ \ \ \ }%
\leq C\left\Vert g_{i}\right\Vert _{q,\infty}\left\vert X_{t_{i}%
,s}(x)-x\right\vert \int_{0}^{1}\psi_{k}(\lambda x+(1-\lambda)X_{t_{i}%
,s}(x))d\lambda.
\end{array}
\right.
\]
It follows%
\[
\left\vert \varepsilon_{i}(x)\right\vert \leq CQ_{q}(T,\widehat{\mathcal{P}%
})\left\Vert f\right\Vert _{q,\infty}\int_{t_{i}}^{t_{i+1}}\int_{0}%
^{1}\mathbb{E}\left[  \psi_{k}(\lambda x+(1-\lambda)X_{t_{i},s}(x))\left\vert
X_{t_{i},s}(x)-x\right\vert \right]  d\lambda ds.
\]
Using the standard estimates on the trajectory, $\mathbb{E}\left[  \left\vert
X_{t_{i},s}(x)\right\vert ^{k}\right]  \leq C\left(  1+\left\vert x\right\vert
^{k}\right)  .$ Hence,%
\[
\mathbb{E}\left[  \psi_{k}^{2}\left(  \lambda x+(1-\lambda)X_{t_{i}%
,s}(x)\right)  \right]  \leq C\psi_{k}^{2}(x).
\]
Using Cauchy-Schwartz inequality, we get%
\[
\frac{1}{\psi_{k}(x)}\left\vert \varepsilon_{i}(x)\right\vert \leq
CQ_{q}(T,\widehat{\mathcal{P}})\left\Vert f\right\Vert _{q,\infty}\int_{t_{i}%
}^{t_{i+1}}\left(  \mathbb{E}\left[  \left\vert X_{t_{i},s}(x)-x\right\vert
^{2}\right]  \right)  ^{\frac{1}{2}}ds.
\]
By setting $\tau_{n}(s):=t_{i}$ for $t_{i}\leq s<t_{i+1},$ we finally get
\[
\frac{1}{\psi_{k}(x)}\sum_{i=1}^{n}\left\vert \varepsilon_{i}(x)\right\vert
\leq CQ_{q}(T,\widehat{\mathcal{P}})\left\Vert f\right\Vert _{q,\infty}%
\int_{0}^{t}\left(  \mathbb{E}\left[  \left\vert X_{\tau_{n}(s),s}%
(x)-x\right\vert ^{2}\right]  \right)  ^{\frac{1}{2}}ds
\]
and the right-hand term vanishes as $n\rightarrow\infty$. Similar estimates
are valid for $\widehat{\varepsilon}$ which concludes our proof.
\end{proof}

\begin{remark}
1. This assertion is to be interpreted in connection to Trotter-Kato-type
results (e.g. \cite[Theorem 4.4]{Pazy1983}) stating that, given $\mathcal{P}%
_{t}$ and $\left(  \mathcal{P}_{t}^{n}\right)  _{n\in%
\mathbb{N}
}$ homogeneous Feller semigroups of infinitesimal operators $\mathcal{L}$
respectively $\mathcal{L}_{n}$, if $\mathcal{L}_{n}$ converges to
$\mathcal{L}$, then $\mathcal{P}_{t}^{n}\rightarrow\mathcal{P}_{t}$ (in an
appropriate sense). The inequality (\ref{dist6}) gives not only qualitative
behavior, but a quantitative one by providing estimate of the error within our framework.

2. The main difficulty and novelty in our approach is to provide (\ref{D19}).
Whenever $\gamma$ is constant, one deals with a usual SDE with jumps and the
proof of (\ref{D19}) follows from the regularity of the flow $x\rightarrow
X_{t}(x).$ However, since $\gamma(t,z,x)$ depends on $x$ (which is the case
for $PDMP$)$,$ the effort developed in the previous sections is necessary.
Note however that (\ref{D19}) is needed only on one of $\mathcal{P}_{t}$ and
$\widehat{\mathcal{P}}_{t}.$ Hence, in a framework in which either $\gamma$ or
$\widehat{\gamma}$ does not depend on $x,$ the proofs simplify considerably.
\end{remark}

\section{PDMP With Three Regimes\label{Section3Regimes}}

In this section we discus piecewise diffusive Markov processes in which three
regimes are at work depending on the speed of the jumps. The intermediate
regime will be purely deterministic and replaced by a drift term
(corresponding to an application of the Law of Large Numbers). The fast regime
will provide a diffusive term (associated with an application of the Central
Limit Theorem). Finally, the slow regime is kept as jump-type contribution. We
do not aim at treating a completely general framework but only at presenting
an example in order to illustrate our approach.

\subsection{Theoretical Framework}

\subsubsection{The Model}

Let us begin with fixing $\varepsilon>0$ and a measurable space
$(E,\mathcal{E},\mu_{\varepsilon})$ where $\mu_{\varepsilon}$ is a non
negative finite measure. The space decomposes as follows $E=A_{\varepsilon
}\cup B_{\varepsilon}\cup C_{\varepsilon},$ where $A_{\varepsilon
},B_{\varepsilon},C_{\varepsilon}$ are mutually disjoint Borel measurable
sets. Moreover, given $\Gamma_{\varepsilon}>0,$ to some (smooth,
time-homogeneous) coefficients $c_{\varepsilon},\gamma_{\varepsilon}%
:E\times\mathbb{R}\rightarrow\mathbb{R},$ $b_{\varepsilon}:\mathbb{R}%
\rightarrow\mathbb{R}\mathbf{,}$ we associate the stochastic equation
\begin{equation}
\left.
\begin{array}
[c]{c}%
X_{t}^{\varepsilon}=x+\int_{0}^{t}b_{\varepsilon}(X_{s}^{\varepsilon}%
)ds+\int_{0}^{t}\int_{A_{\varepsilon}\times\lbrack0,2\Gamma_{\varepsilon}%
]}c_{\varepsilon}(z,X_{s-}^{\varepsilon})1_{\{u\leq\gamma_{\varepsilon
}(z,X_{s-}^{\varepsilon})\}}\widetilde{N}_{\mu_{\varepsilon}}(ds,dz,du)\\
+\int_{0}^{t}\int_{(B_{\varepsilon}\cup C_{\varepsilon})\times\lbrack
0,2\Gamma_{\varepsilon}]}c_{\varepsilon}(z,X_{s-}^{\varepsilon})1_{\{u\leq
\gamma_{\varepsilon}(z,X_{s-}^{\varepsilon})\}}N_{\mu_{\varepsilon}}(ds,dz,du)
\end{array}
\right.  \label{r1}%
\end{equation}
where $N_{\mu_{\varepsilon}}$ is a Poisson point measure on $E\times R_{+}$
associated with $\Gamma_{\varepsilon}$ and $\mu_{\varepsilon}$ and
$\widetilde{N}_{\mu_{\varepsilon}}=N_{\mu_{\varepsilon}}-\widehat{N}%
_{\mu_{\varepsilon}}$ is the associated martingale measure$.$

\subsubsection{The Regimes}

The jumps in $A_{\varepsilon}$ are assumed to occur at high frequency. They
lead to a Brownian motion. The jumps in $B_{\varepsilon}$ represent an
intermediary regime which will be modeled by a drift term while the jumps in
$C_{\varepsilon}$ are rather rare and remain in the same regime. This model is
expressed by the following setting. We consider a finite measure $\mu$\ and
the coefficients $\sigma,b:\mathbb{R}\rightarrow\mathbb{R}$ and $c,\gamma
:E\times\mathbb{R}\rightarrow\mathbb{R}.$ We associate the equation
\begin{equation}
X_{t}=x+\int_{0}^{t}b(X_{s})ds+\int_{0}^{t}\sigma(X_{s})dW_{s}+\int_{0}%
^{t}\int_{E\times\lbrack0,2\Gamma]}c(z,X_{s-})1_{\{u\leq\gamma(z,X_{s-}%
)\}}N_{\mu}(ds,dz,du),\text{ 0}\leq t\leq T. \label{r2}%
\end{equation}

Our aim is to give sufficient conditions in order to obtain the convergence of
the family $X^{\varepsilon}$ to $X$ and to estimate the error.

\subsubsection{Standing (Sufficient) Assumptions}

Throughout the section, unless stated otherwise, we assume the following.

\begin{assumption}
[Assumption $H_{0}^{\varepsilon}$]We assume that $c_{\varepsilon}$ and
$\gamma_{\varepsilon}$ satisfy integrability condition (\ref{Not7}), the
Lipschitz regularity assumption (\ref{Not6}) and the uniform upper-bound of
$\gamma_{\varepsilon}$ assumption (\ref{Not5}) (written for $\Gamma
_{\varepsilon}$ substituting $\Gamma)$.
\end{assumption}

\begin{remark}
Note that the constants which appear in these conditions depend on
$\varepsilon$ (so they are not uniform with respect to $\varepsilon).$ Under
these hypothesis, the equation (\ref{r1}) has a unique solution (which may
alternatively be constructed using a compound Poisson process).
\end{remark}

We also need an assumption on the limit coefficients.

\begin{assumption}
[Assumption $H_{0}$]We assume that $\sigma,b\in C_{b}^{3}(\mathbb{R})$ and,
for every $z\in E,$ the functions $x\mapsto c(z,x)$ and $x\mapsto\ln
\gamma(z,x)$ are three times differentiable and
\begin{equation}
\sum_{0\leq\left\vert \alpha\right\vert \leq3}\sup_{x\in\mathbb{R}}\left[
\left\vert \partial^{\alpha}\sigma(x)\right\vert +\left\vert \partial^{\alpha
}b(x)\right\vert +\sup_{z\in E}\left\vert \partial^{\alpha}c(z,x)\right\vert
+\sup_{z\in E}\left\vert \partial^{\alpha}\ln\gamma(z,x)\right\vert \right]
=:C_{\ast}<\infty. \label{r3}%
\end{equation}

\end{assumption}

Under this hypothesis, the equation (\ref{r2}) has a unique solution (see
Remark \ref{Existance}). Finally, we need some further assumptions in order to
obtain convergence$.$ We denote by $\nu_{\varepsilon}(x,dz):=\gamma
_{\varepsilon}(z,x)\mu_{\varepsilon}(dz)$\ and set
\[
\left.
\begin{array}
[c]{l}%
\sigma_{\varepsilon}(x):=\left(  \int_{A_{\varepsilon}}c_{\varepsilon}%
^{2}(z,x)\nu_{\varepsilon}(x,dz)\right)  ^{\frac{1}{2}},\quad b^{\varepsilon
}(x):=b_{\varepsilon}\left(  x\right)  +\int_{B_{\varepsilon}}c_{\varepsilon
}(z,x)\nu_{\varepsilon}(x,dz),\text{ }\\
\delta_{\sigma}(\varepsilon):=\left\Vert \sigma_{\varepsilon}^{2}-\sigma
^{2}\right\Vert _{\infty},\quad\delta_{b}(\varepsilon)=\left\Vert
b_{\varepsilon}-b\right\Vert _{\infty},\text{ }\\
\delta_{c,\gamma}(\varepsilon)=\underset{x\in%
\mathbb{R}
}{\sup}\int_{C_{\varepsilon}}\left\vert (c-c_{\varepsilon})(z,x)\right\vert
\gamma(z,x)+\left\vert (\gamma-\gamma_{\varepsilon})(z,x)\right\vert d\mu(z).
\end{array}
\right.
\]
Moreover we denote the convenient moments by%
\[
\left.
\begin{array}
[c]{l}%
\delta_{A}(\varepsilon)=\underset{x\in%
\mathbb{R}
}{\sup}\int_{A_{\varepsilon}}\left\vert c_{\varepsilon}(z,x)\right\vert
^{3}\nu_{\varepsilon}(x,dz),\quad\delta_{B}(\varepsilon)=\underset{x\in%
\mathbb{R}
}{\sup}\int_{B_{\varepsilon}}\left\vert c_{\varepsilon}(z,x)\right\vert
^{2}\nu_{\varepsilon}(x,dz),\text{ }\\
\delta_{C}(\varepsilon)=\underset{x\in%
\mathbb{R}
}{\sup}\int_{E-C_{\varepsilon}}\left\vert c(z,x)\right\vert \gamma(z,x)\mu(dz)
\end{array}
\right.
\]

\begin{assumption}
[Assumption $H_{1}$]We assume that $\delta(\varepsilon):=\delta_{\sigma
}(\varepsilon)+\delta_{b}(\varepsilon)+\delta_{c,\gamma}(\varepsilon
)+\delta_{A}(\varepsilon)+\delta_{B}(\varepsilon)+\delta_{C}(\varepsilon
)\underset{\varepsilon\rightarrow0}{\rightarrow}0.$
\end{assumption}

\begin{assumption}
[Assumption $H_{2}$]Finally, we assume that the restrictions of $\mu
_{\varepsilon}$ and $\mu$ to $C_{\varepsilon}$ coincide, i.e.
$1_{C_{\varepsilon}}(z)\mu_{\varepsilon}(dz)=1_{C_{\varepsilon}}(z)\mu(dz).$
\end{assumption}

\subsubsection{The Theoretical Result}

Under these assumptions, one can state and prove the following.

\begin{theorem}
\label{ThExp}We assume that $H_{0}^{\varepsilon},H_{0},H_{1}$ and $H_{2}$ hold
true. We let $\mathcal{P}_{t}^{\varepsilon}$ and $\mathcal{P}_{t}$\ be the
semigroups associated with $X_{t}^{\varepsilon}$ respectively with $X_{t}.$
Then, there exists a universal constant $C$ such that, for every $f\in
C_{b}^{3}(R),$
\begin{equation}
\left\Vert \mathcal{P}_{t}^{\varepsilon}f-\mathcal{P}_{t}f\right\Vert
_{\infty}\leq(t\vee1)^{3}CC_{\ast}^{42}\exp((t\vee1)CC_{\ast}^{36}%
)\times\delta(\varepsilon)\left\Vert f\right\Vert _{3,\infty}. \label{r6}%
\end{equation}

\end{theorem}

The proof follows from Theorem \ref{Conv} and, for our readers' convenience a
sketch is presented in Section \ref{SectionProof3Regimes}.

\begin{remark}
The notation used in the previous theorem suggests that $\mathcal{P}%
_{t}^{\varepsilon}$ is an approximation of $\mathcal{P}_{t}.$ However,
sometimes, the point of view is the exact opposite: the physical phenomenon is
modeled by $X_{t}^{\varepsilon}$ and $X_{t}$ represents an approximation which
is easier to handle. Having this in mind one may also consider the following
optimization problem: given the dynamics of $X_{t}^{\varepsilon},$ which is
the best dynamics (coefficients) of type $X_{t}$ which approximates
$X_{t}^{\varepsilon}?$ In order to formulate this problem in a clean way one
has to give a criterion in order to precise the sense of "best". This would be
another problem that escapes the aim of the present paper.
\end{remark}

\subsection{A Simple Example\label{SubsectionSimpleExp}}

Let us now give an explicit example.

\begin{example}
To this purpose, we consider $c,\gamma\in C_{b}^{3}(\mathbb{R})$ and
\begin{align*}
\mu_{\varepsilon}(dz)  &  =1_{(\varepsilon,3\varepsilon]}(z)\frac{dz}{z^{2}%
}+1_{(3\varepsilon,4\varepsilon]}(z)\frac{dz}{z^{3/2}}+1_{(4\varepsilon
,1]}(z)\frac{dz}{z},\\
c_{\varepsilon}(z,x)  &  =c(x)\sqrt{z}\left(  1_{(2\varepsilon,1]}%
(z)-\alpha1_{(\varepsilon,2\varepsilon]}(z)\right)  \text{ with }\alpha
=\frac{\sqrt{3}-\sqrt{2}}{\sqrt{6}-\sqrt{3}}.
\end{align*}
and we associate the equation
\[
X_{t}^{\varepsilon}=x+\int_{0}^{t}\int_{0}^{1}\int_{0}^{1}c_{\varepsilon
}(z,X_{s-}^{\varepsilon})1_{\{u\leq\gamma(X_{s-}^{\varepsilon})\}}%
N_{\mu_{\varepsilon}}(ds,dz,du).
\]
Note that, in contrast with the equation (\ref{r1}), the measure
$N_{\mu_{\varepsilon}}$ is not compensated. But, in fact, the activity of the
small jumps in $1_{(\varepsilon,2\varepsilon]}(z)$ compensate the activity of
the small jumps in $1_{(2\varepsilon,3\varepsilon]}(z).$ The limit equation is%
\[
X_{t}=x+\int_{0}^{t}\sigma(X_{s})dW_{s}+\int_{0}^{t}b(X_{s})ds+\int_{0}%
^{t}\int_{0}^{1}\int_{0}^{1}c(X_{s-})\sqrt{z}1_{\{u\leq\gamma(X_{s-}%
^{\varepsilon})\}}N_{\mu}(ds,dz,du)
\]
with $\mu(dz)=z^{-1}dz$ and $\sigma(x)=\beta_{1}c(x)\sqrt{\gamma}%
(x),b(x)=\int_{B_{\varepsilon}}c_{\varepsilon}(z,x)\nu_{\varepsilon
}(x,dz)=\beta_{2}c(x)\gamma(x).$ Here, $\beta_{1}=\left(  (\alpha^{2}%
-1)\ln2+\ln3\right)  ^{\frac{1}{2}}$ and $\beta_{2}=\ln\frac{4}{3}.$ Then, by
applying Theorem \ref{ThExp}, it follows that
\[
\left\Vert \mathcal{P}_{t}f-\mathcal{P}_{t}^{\varepsilon}f\right\Vert
_{\infty}\leq C\sqrt{\varepsilon}\left\Vert f\right\Vert _{3,\infty}%
\]
with $C$ depending on $\left\Vert c\right\Vert _{3,\infty}$ and $\left\Vert
\ln\gamma\right\Vert _{3,\infty}.$ To this purpose, one only needs to check
the assumption $H_{1}$ (see Section \ref{SectionProof3Regimes} for details on
this step) and apply Theorem \ref{ThExp}.
\end{example}

\section{Boltzmann's equation\label{SectionBoltzmann}}

\subsection{The Equation}

\subsubsection{The Model}

In this section we use the previous results to construct an approximation
scheme for the solution of the two-dimensional Boltzmann equation taking the
following form
\begin{equation}
\partial_{t}f_{t}(v)=\int_{\mathbb{R}^{2}}dv_{\ast}\int_{-\pi/2}^{\pi
/2}d\theta\left\vert v-v_{\ast}\right\vert ^{\kappa}\theta^{-(1+\nu)}%
(f_{t}(v^{\prime})f_{t}(v_{\ast}^{\prime})-f_{t}(v)f_{t}(v_{\ast})).
\label{bo1}%
\end{equation}
Here,

\begin{itemize}
\item $f_{t}(v)$ is a non negative measure on $\mathbb{R}^{2}$ representing
the density of particles with velocity $v$ in a model for a gas in dimension two.

\item $R_{\theta}$ is the rotation of angle $\theta$ and the new speeds after
collision are $v^{\prime}=\frac{v+v_{\ast}}{2}+R_{\theta}\left(
\frac{v-v_{\ast}}{2}\right)  $ respectively $v_{\ast}^{\prime}=\frac
{v+v_{\ast}}{2}-R_{\theta}\left(  \frac{v-v_{\ast}}{2}\right)  .$

\item the parameters $\nu\in(0,1)$ and $\kappa\in(0,1]$ are chosen for the
cross section to model the interaction in the spirit of the assumption
\textbf{A}$\left(  \gamma,\nu\right)  $ in \cite{BallyFournier2011}.
\end{itemize}

The rigorous sense of this equation is given by integrating it against a test
function (hence leading to weak solutions of (\ref{bo1})). In \cite[Corollary
2.3 and Lemma 4.1]{FournierMouhot2009}, the authors have proven that, for
every $\nu\in(0,1)$ and $\kappa\in(0,1],$ the above equation admits a unique
weak solution as follows. One assumes that there exists $s\in(\kappa,2)$ such
that $\int e^{\left\vert v\right\vert ^{s}}f_{0}(dv)<\infty.$ Then, there
exists a unique solution $f_{t}$ of (\ref{bo1}) which starts from $f_{0}.$
Moreover, the solution satisfies $\underset{t\leq T}{\sup}\int e^{\left\vert
v\right\vert ^{s^{\prime}}}f_{t}(dv)<\infty$ for every $s^{\prime}<s.$

Using Skorohod representation theorem, we find a measurable function
$v_{t}:[0,1]\rightarrow\mathbb{R}^{2}$ such that for every $\psi
:\mathbb{R}^{2}\rightarrow\mathbb{R}_{+}$%
\begin{equation}
\int_{0}^{1}\psi(v_{t}(\rho))d\rho=\int_{R^{2}}\psi(v)f_{t}(dv). \label{bo2}%
\end{equation}

Throughout the section, unless stated otherwise, we fix $\nu,\kappa$ and
$s\in(\kappa,2)$ and the corresponding solution $f_{t}(v)$ (and, in
particular, $v_{t}(\rho)).$

\subsubsection{Probabilistic Interpretation. Approximations}

In \cite{Tanaka1978}, the author gives a probabilistic interpretation for the
solutions of the classical Bolzmann equation (in dimension 3). A variant of
this result in dimension two (so for the equation (\ref{bo1})), as well as an
approximation result for it, is given in \cite[Section 2]{BallyFournier2011}.
We briefly recall these elements.

We let $\left(  \Omega,\mathcal{F},\mathbb{P}\right)  $ be a probability
space, the space $E:=\left[  -\frac{\pi}{2},\frac{\pi}{2}\right]
\times\lbrack0,1]$ and let $N(dt,d\theta,d\rho,du)$ be a Poisson point measure
on $E\times\mathbb{R}_{+}$ with intensity measure $\theta^{-(1+\nu)}%
d\theta\times d\rho\times du.$ We also consider the matrix
\[
A(\theta):=\frac{1}{2}\left(
\begin{tabular}
[c]{ll}%
$\cos\theta-1$ & $-\sin\theta$\\
$\sin\theta$ & $\cos\theta-1$%
\end{tabular}
\ \right)  =\frac{1}{2}(R_{\theta}-I).
\]
Then we are interested in the equation
\begin{equation}
V_{t}=V_{0}+\int_{0}^{t}\int_{E\times\mathbb{R}_{+}}A(\theta)(V_{s-}%
-v_{s}(\rho))1_{\{u\leq\left\vert V_{s-}-v_{s}(\rho)\right\vert ^{\kappa}%
\}}N(ds,d\theta,d\rho,du) \label{bo3}%
\end{equation}
with $\mathbb{P}(V_{0}\in dv)=f_{0}(dv).$

In the spirit of \cite[Section 2]{BallyFournier2011}, one also constructs the
following approximation. One considers a $C^{\infty}$ even non-negative
function $\chi$ supported by $[-1,1]$ and such that $\int_{R}\chi(x)dx=1.$ We
fix $\eta_{0}\in\left(  \frac{1}{s},\frac{1}{\kappa\vee\nu}\right)  .$ Given
$\varepsilon\in(0,1],$ we denote by $\Gamma_{\varepsilon}=\left(  \ln\frac
{1}{\varepsilon}\right)  ^{\eta_{0}}$ and define%
\begin{equation}
\varphi_{\varepsilon}(x)=\int_{R}((y\vee2\varepsilon)\wedge\Gamma
_{\varepsilon})\frac{\chi(\frac{x-y}{\varepsilon})}{\varepsilon}dy.
\label{bo5}%
\end{equation}
The reader is invited to note that $2\varepsilon\leq\varphi_{\varepsilon
}(x)\leq\Gamma_{\varepsilon},$ for every $x\in\mathbb{R},\varphi_{\varepsilon
}(x)=x,$ for $x\in(3\varepsilon,\Gamma_{\varepsilon}-1),\varphi_{\varepsilon
}(x)=2\varepsilon$ for $x\in(0,\varepsilon)$ and $\varphi_{\varepsilon
}(x)=\Gamma_{\varepsilon}$ for $x\in(\Gamma_{\varepsilon},\infty).$

To the cut off function $\varphi_{\varepsilon},$ one associates the equation
\begin{equation}
V_{t}^{\varepsilon}=V_{0}+\int_{0}^{t}\int_{E\times R_{+}}A(\theta
)(V_{s-}^{\varepsilon}-v_{s}(\rho))1_{\{u\leq\varphi_{\varepsilon}^{\kappa
}(\left\vert V_{s-}^{\varepsilon}-v_{s}(\rho)\right\vert )\}}N(ds,d\theta
,d\rho,du). \label{bo7}%
\end{equation}
Proposition 2.1 in \cite{BallyFournier2011} provides the following
probabilistic interpretation as well as an approximation result.

\begin{proposition}
[{\cite[Proposition 2.1]{BallyFournier2011}}]1. The equation (\ref{bo3}) has a
unique c\`{a}dl\`{a}g adapted solution $(V_{t})_{t\geq0}$ and its law
$\mathbb{P}(V_{t}\in dv)=f_{t}(dv)$\footnote{In this sense, $V_{t}$ provides a
probabilistic representation for $f_{t}$}$.$

2. The equation (\ref{bo7}) has a unique c\`{a}dl\`{a}g solution
$V^{\varepsilon}$ and
\begin{equation}
\sup_{t\leq T}\mathbb{E}\left[  \left\vert V_{t}-V_{t}^{\varepsilon
}\right\vert \right]  \leq Ce^{C\Gamma_{\varepsilon}^{\kappa}}\times
\varepsilon^{1+\kappa}. \label{bo8}%
\end{equation}
Moreover, there exists $\varepsilon_{0}>0$ such that, for every $0<s^{\prime
}<s,$%
\begin{equation}
\sup_{\varepsilon\leq\varepsilon_{0}}\mathbb{E}\left[  \sup_{t\leq T}\left(
e^{\left\vert V_{t}\right\vert ^{s^{\prime}}}+e^{\left\vert V_{t}%
^{\varepsilon}\right\vert ^{s^{\prime}}}\right)  \right]  <\infty. \label{bo4}%
\end{equation}

\end{proposition}

In the following, we assume that
\begin{equation}
\int e^{\left\vert v\right\vert ^{s}}f_{0}(dv)<\infty\quad\forall s<2.
\label{bo4'}%
\end{equation}
In particular this gives the restriction $\frac{1}{2}<\eta_{0}<\frac{1}%
{\kappa\vee\nu}.$ Then, if $p\geq1$ is such that $\kappa p<2,$ we may choose
$\eta_{0}$ such that $\eta_{0}\kappa p<1.$ This guarantees that for every
$a>0$ there exists $\varepsilon_{a}$ (small enough) such that, with
$a(\varepsilon)=\left(  \ln\frac{1}{\varepsilon}\right)  ^{-1+\eta_{0}\kappa
p},$
\begin{equation}
e^{\Gamma_{\varepsilon}^{\kappa p}}=\varepsilon^{-a(\varepsilon)}%
\leq\varepsilon^{-a}\quad\forall0<\varepsilon<\varepsilon_{a}. \label{bo4''}%
\end{equation}

\subsection{First Order Approximation}

The aim of this section is to construct an approximation of the solution
$V_{t}$ of the equation (\ref{bo3}) in which the small jumps, corresponding to
$\left\vert \theta\right\vert \leq\delta,$ are replaced by a drift term.

First, let us fix $\delta>0,\quad r=\frac{2-3\nu}{3+\kappa}\left(  \leq
\frac{1-\nu}{1-\kappa}\right)  $, set $\varepsilon=\delta^{r}$ and consider
the solution $V_{t}^{\varepsilon}$ of the truncated equation (\ref{bo7})
associated with this $\varepsilon.$ The inequality (\ref{bo8}) provides a
control of the distance between $V_{t}$ and $V_{t}^{\varepsilon}.$ Second, for
this solution $V_{t}^{\varepsilon}$ of (\ref{bo7}) we apply Lemma \ref{Conv}
in order to replace the small jumps by a convenient drift term.

To fall in the framework given in the first part of our paper, we will denote
by $V_{t_{0},t}^{\varepsilon}(v)$ the solution of the equation (\ref{bo7})
which starts from $v\in%
\mathbb{R}
^{2}$ at time $t_{0}\in\left[  0,T\right]  $ and we set $\mathcal{P}_{t_{0}%
,t}^{\varepsilon}f(v)=\mathbb{E}\left[  f\left(  V_{t_{0},t}^{\varepsilon
}(v)\right)  \right]  .$ We also denote (for $\varepsilon>0$ fixed above),%
\[
\mu(d\theta,d\rho)=\theta^{-(1+\nu)}d\theta\times d\rho,\quad c(t,\theta
,\rho,v)=A(\theta)(v-v_{t}(\rho)),\quad\gamma_{\varepsilon}(t,\rho
,v)=\varphi_{\varepsilon}^{\kappa}(\left\vert v-v_{t}(\rho)\right\vert ).
\]
The infinitesimal operator of $\mathcal{P}_{t_{0},t}^{\varepsilon}$ is simply
given by
\[
\mathcal{L}_{t}^{\varepsilon}f(v)=\int_{E}\mu(d\theta,d\rho)\gamma
(t,\rho,v)(f(v+c(t,\theta,\rho,v))-f(v)).
\]

We will replace the activity of small jumps (such that $\theta$ is close to
$0$) with a drift term. To this purpose, we denote by $E_{\delta}%
=\{(\theta,\rho):\left\vert \theta\right\vert >\delta\}$ and we define%
\begin{equation}
\left.
\begin{array}
[c]{l}%
b_{\delta}(t,v)=\int_{\{\left\vert \theta\right\vert \leq\delta\}}%
\gamma(t,\rho,v)c(t,\theta,\rho,v)\mu(d\theta,d\rho)\text{ and}\\
\widehat{\mathcal{L}}_{t}^{\delta}f(v)=b_{\delta}(t,v)\partial f(v)+\int
_{E_{\delta}}\mu(d\theta,d\rho)\gamma(t,\rho,v)(f(v+c(t,\theta,\rho,v))-f(v)).
\end{array}
\right.  \label{b1}%
\end{equation}
The approximating equation is
\[
U_{t_{0},t}^{\delta}(v)=v+\int_{t_{0}}^{t}b_{\delta}(s,U_{t_{0},s}^{\delta
}(v))ds+\int_{t_{0}}^{t}\int_{E_{\delta}\times R_{+}}c(s,\theta,\rho
,U_{t_{0},s-}^{\delta}(v))1_{\{u\leq\gamma(s,\rho,U_{t_{0},s-}^{\delta}%
(v))\}}N(ds,d\theta,d\rho,du).
\]
We denote by $\widehat{\mathcal{P}}_{t_{0},t}^{\delta}$ the semigroup
associated with $\widehat{\mathcal{L}}_{t}^{\delta},$ that is $\widehat
{\mathcal{P}}_{t_{0},t}^{\delta}f(v):=\mathbb{E}\left[  f\left(  U_{t_{0}%
,s}^{\delta}(v)\right)  \right]  $.\ 

\begin{theorem}
\label{ThBoltzmannOrder1}Suppose that $\kappa<\frac{1}{8}$ and $\nu<\frac
{1}{2}.$ For every $\eta<\frac{(2-3\nu)(1+\kappa)}{3+\kappa}$\ there exists
$C_{\eta}\geq1$ and $\delta_{\eta}>0$ such that for $0<\delta\leq\delta_{\eta
}$\ we have
\begin{equation}
\left\vert \mathbb{E}\left[  f(V_{t})\right]  -\mathbb{E}\left[  f\left(
U^{\delta}(V_{0})\right)  \right]  \right\vert \leq C_{\eta}\left\Vert
f\right\Vert _{2,\infty}\times\delta^{\eta}. \label{bo12}%
\end{equation}

\end{theorem}

The proof essentially consists in the use of Theorem \ref{Conv} combined with
(\ref{bo8}). For our readers' sake, the complete proof is given in Section
\ref{SectionProofBoltzmann}.

\subsection{Second Order Approximation}

We define%
\[
\left.
\begin{array}
[c]{l}%
\sigma(t,\theta,\rho,v)=c(t,\theta,\rho,v)\sqrt{\gamma_{\varepsilon}%
(t,\rho,v)},\quad a_{\delta}^{i,j}(t,v)=\int_{\{\left\vert \theta\right\vert
\leq\delta\}}\mu(d\theta,d\rho)\sigma^{i}\sigma^{j}(t,\theta,\rho,v),\\
\widehat{\mathcal{L}}_{t}^{\delta}f(v)=\left\langle b_{\delta}(t,v),\nabla
f(v)\right\rangle +\frac{1}{2}\sum_{i,j=1}^{d}a_{\delta}^{i,j}(t,v)\partial
_{ji}^{2}f(v)\\
\text{ \ \ \ \ \ \ }+\int_{E_{\delta}}\mu(d\theta,d\rho)\gamma_{\varepsilon
}(t,\rho,v)(f(v+c(t,\theta,\rho,v))-f(v)).
\end{array}
\right.
\]
where $b_{\delta}$ is given by (\ref{b1}). This is the infinitesimal operator
corresponding to the semigroup $\widehat{\mathcal{P}}_{t_{0},t}^{\delta
}f(v)=\mathbb{E}\left[  f(U_{t_{0},t}^{\delta}(v))\right]  $ with $U_{t_{0}%
,t}^{\delta}(v)$ solution to
\[
\left.
\begin{array}
[c]{l}%
U_{t_{0},t}^{\delta}(v)=v+\int_{t_{0}}^{t}b_{\delta}\left(  s,U_{t_{0}%
,s}^{\delta}(v)\right)  ds+\int_{0}^{t}\int_{E_{\delta}}\sigma_{\delta
}(s,\theta,\rho,U_{t_{0},s-}^{\delta}(v))W_{\mu}(ds,d\theta,d\rho)\\
\text{ \ \ \ \ \ \ \ \ \ \ }+\int_{0}^{t}\int_{E_{\delta}\times R_{+}%
}c(s,\theta,\rho,U_{t_{0},s-}^{\delta}(v))1_{\{u\leq\gamma(s,\rho,U_{t_{0}%
,s-}^{\delta}(v))\}}N(ds,d\theta,d\rho,du).
\end{array}
\right.
\]
The approach is quite similar to the first order. The main result is the following.

\begin{theorem}
\label{ThBoltzmann2}Let us assume that\textbf{\ }$\kappa\leq\frac{1}{18}$ and
let
\begin{equation}
r<\frac{1-\nu}{2-\kappa}\wedge\frac{1-\frac{\nu}{2}}{2-\frac{\kappa}{2}}%
\wedge\frac{3-4\nu}{4+\kappa}. \label{L7}%
\end{equation}
Then%
\begin{equation}
\left\Vert \frac{1}{\psi_{3}}(\mathcal{P}_{t_{0},t}f-\widehat{\mathcal{P}%
}_{t_{0},t}^{\delta})f\right\Vert _{\infty}\leq C\delta^{r(1+\kappa)}%
\times\left\Vert f\right\Vert _{3,\infty}. \label{L5}%
\end{equation}

\end{theorem}

\begin{remark}
It turns out that the second order error is larger then the first order error.
This is somewhat counterintuitive. This is due to the fact that we mix two
different errors: $\mathcal{P}_{t_{0},t}f-\mathcal{P}_{t_{0},t}^{\varepsilon
}\sim\varepsilon^{1+\kappa}$ and $\mathcal{P}_{t_{0},t}^{\varepsilon
}f-\widehat{\mathcal{P}}_{t_{0},t}^{\delta}\sim\delta^{3-\nu}\varepsilon
^{-3}.$ If $\varepsilon$ \ is fixed then the second order error is
$\delta^{3-\nu}$ and the first order error is $\delta^{2-\nu}$ and this seems
coherent. But if we mix the two errors things become less obvious.
\end{remark}

\section{Proof of the Results}

\subsection{Proof of the Results in Section \ref{SectHybridSyst}%
\label{ProofLocalizationLemma}}

We begin with the estimates given in Lemma \ref{Distance}.

\begin{proof}
[Proof of Lemma \ref{Distance}]We follow the ideas in \cite{Graham1992} so we
just sketch the proof. Let us fix the initial time $s<T.$ For every
$t\in\left[  s,T\right]  ,$ one has%
\[
\left\vert \Delta X_{s,t}\right\vert \leq\left\vert \Delta X_{s,s}\right\vert
+\left\vert \int_{s}^{t}\int_{E}h(r,z)dW_{\mu}(dr,dz)\right\vert +\left\vert
\int_{s}^{t}g(r)dr\right\vert +\left\vert \int_{s}^{t}\int_{E\times
\lbrack0,2\Gamma]}H(r-,z,u)N_{\mu}(dr,dz,du)\right\vert
\]
where%
\[
h(r,z)=\sigma(r,z,X_{s,r}^{G_{1}})-\sigma(r,z,X_{s,r}^{G_{2}}),\quad
g(r)=b(r,X_{s,r}^{G_{1}})-b(r,X_{s,r}^{G_{2}})
\]
and%
\[
H(r,z,u)=1_{G_{1}}(z)c(r,z,X_{s,r-}^{G_{1}})1_{\{u\leq\gamma(r,z,X_{s,r-}%
^{G_{1}})\}}-1_{G_{2}}(z)c(r,z,X_{s,r-}^{G_{2}})1_{\{u\leq\gamma
(r,z,X_{s,r-}^{G_{2}})\}},
\]
for all $\left(  r,z,u\right)  \in\left[  s,t\right]  \times E\times%
\mathbb{R}
_{+}$. Using the inequality
\[
\left\vert h(r,z)\right\vert \leq\left\vert \Delta X_{s,r}\right\vert
\times\int_{0}^{1}\left\vert \nabla\sigma(r,z,\lambda X_{s,r-}^{G_{1}%
}+(1-\lambda)X_{s,r-}^{G_{2}})\right\vert d\lambda,
\]
we obtain%
\[
\left[  \int_{E}\left\vert h(r,z)\right\vert ^{2}\mu(dz)\right]  ^{\frac{1}%
{2}}\leq\left\Vert \nabla\sigma\right\Vert _{(\mu,\infty)}\left\vert \Delta
X_{s,r}\right\vert .
\]
Burkholder's inequality yields%
\begin{align*}
&  \mathbb{E}\left[  \sup_{s\leq t^{\prime}\leq t}\left\vert \int
_{s}^{t^{\prime}}\int_{E}h(r,z)dW_{\mu}(dr,dz)\right\vert \right]  \leq
C\mathbb{E}\left[  \left(  \int_{s}^{t}\int_{E}\left\vert h(r,z)\right\vert
^{2}\mu(dz)dr\right)  ^{\frac{1}{2}}\right] \\
&  \leq C\left\Vert \nabla\sigma\right\Vert _{(\mu,\infty)}\mathbb{E}\left[
\left(  \int_{s}^{t}\left\vert \Delta X_{s,r}\right\vert ^{2}dr\right)
^{\frac{1}{2}}\right]  \leq C\left\Vert \nabla\sigma\right\Vert _{(\mu
,\infty)}\left(  t-s\right)  ^{\frac{1}{2}}\mathbb{E}\left[  \sup_{s\leq r\leq
t}\left\vert \Delta X_{s,r}\right\vert \right]  .
\end{align*}
And the same inequality holds for $g.$ Finally, since $N_{\mu}$ is a positive
measure, one has%
\[
\mathbb{E}\left[  \sup_{s\leq t^{\prime}\leq t}\left\vert \int_{s}^{t^{\prime
}}\int_{E\times\lbrack0,2\Gamma]}H(r-,z,u)N_{\mu}(drdzdu)\right\vert \right]
\leq\mathbb{E}\left[  \int_{s}^{T}dr\int_{E\times\lbrack0,2\Gamma]}\left\vert
H(r,z,u)\right\vert \mu(dz)du\right]  .
\]
A careful analysis of the term $\left\vert H(r,z,u)\right\vert $ shows that
the above term is upper bounded by $(t-s)\alpha(G_{2}\diagdown G_{1}%
)+C(c,\gamma)\int_{s}^{t}\mathbb{E}\left[  \left\vert \Delta X_{s,r}%
\right\vert \right]  dr.$ Going back to the initial inequality in our proof,
one gets%
\begin{align*}
\mathbb{E}\left[  \sup_{s\leq r\leq t}\left\vert \Delta X_{s,r}\right\vert
\right]   &  \leq\left\vert \Delta X_{s,s}\right\vert +(t-s)\alpha
(G_{2}\diagdown G_{1})\\
&  +C(\left\Vert \nabla\sigma\right\Vert _{(\mu,\infty)}+\left\Vert \nabla
b\right\Vert _{\infty}+C(\gamma,c)))\left(  t-s\right)  ^{\frac{1}{2}%
}\mathbb{E}\left[  \sup_{s\leq r\leq t}\left\vert \Delta X_{s,r}\right\vert
\right]  .
\end{align*}
Hence, whenever $t-s\leq\delta:=\left(  2C(\left\Vert \nabla\sigma\right\Vert
_{(\mu,\infty)}+\left\Vert \nabla b\right\Vert _{\infty}+C(\gamma,c)))\right)
^{-2}$, one gets
\[
\mathbb{E}\left[  \sup_{s\leq r\leq t}\left\vert \Delta X_{s,r}\right\vert
\right]  \leq2\left(  \left\vert \Delta X_{s,s}\right\vert +(t-s)\alpha
(G_{2}\diagdown G_{1})\right)  .
\]
The argument follows by partitioning $\left[  s,T\right]  $ in $n\leq4T\left(
C(\left\Vert \nabla\sigma\right\Vert _{(\mu,\infty)}+\left\Vert \nabla
b\right\Vert _{\infty}+C(\gamma,c)))\right)  ^{2}+1$ subintervals of length
$\delta$ and iterating.
\end{proof}

\subsection{Proof of the Results in Section \ref{SectDifferentiability}}

The proof of Lemma \ref{EstimatesPD} makes extensive use of moment estimates
of some kind of linear-type stochastic system. To this purpose, we begin with
briefly explaining the type of system and the estimates we have in mind.

\subsubsection{Preliminary Arguments for Lemma \ref{EstimatesPD}: Moment
Estimates for Linear SDE}

In this section we consider the $d-$dimensional linear equation
\begin{align}
V_{t}  &  =V_{0}+\int_{0}^{t}\int_{E}(h(s)+\left\langle \nabla b(s,X_{s}%
),V_{s}\right\rangle )ds\label{M1}\\
&  +\int_{0}^{t}\int_{E}(H(s,z)+\left\langle \nabla\sigma(s,z,X_{s}%
),V_{s}\right\rangle )W_{\mu}(ds,dz)\nonumber\\
&  +\int_{0}^{t}\int_{G\times(0,2\Gamma)}(Q(s-,z)+\left\langle \nabla
_{x}c(s,z,X_{s-}),V_{s-}\right\rangle )1_{\{u\leq\gamma(s,z,X_{s-})\}}N_{\mu
}(ds,du,dz)\nonumber
\end{align}
Here $X_{s}$ is the solution of the equation (\ref{I1}) and $H,h$ and $Q$ are
predictable processes which verify%
\[
\mathbb{E}\left[  \int_{0}^{T}\left(  \left\Vert H(s,\cdot)\right\Vert
_{\mathbb{L}^{2}(\mu)}^{2}+\left\vert h(s)\right\vert \right)  ds+\sup_{s\leq
T}\sup_{x\in\mathbb{R}^{d}}\int_{G}\left\vert Q(s,z)\right\vert \gamma
(s,z,x)\mu(dz))\right]  <\infty.
\]
This type of condition is needed in order for the corresponding stochastic
(respectively Lebesgue) integrals in (\ref{M1}) to make sense.

\begin{proposition}
\label{Moments}We assume that there exists some predictable process $R$ and
some measurable function $\rho:\mathbb{R}_{+}\times E\times\mathbb{R}%
^{d}\rightarrow\mathbb{R}_{+}$ such that
\begin{equation}
\left\vert Q(s,z)\right\vert \leq\rho(s,z,X_{s})\left\vert R_{s}\right\vert ,
\label{M2}%
\end{equation}
$\mathbb{P}$-almost everywhere on $\Omega$, for all $\left(  s,z\right)
\in\mathbb{R}_{+}\times E.$ Then, for every $p\geq2$ there exists a universal
constant $C$ (depending on $p$ but not on the coefficients) such
that\footnote{We recall that $\left\vert \rho\right\vert _{G,p}$ is defined in
(\ref{D5}) and $[\rho]_{G,p}=\sup_{1\leq p^{\prime}\leq p}\left\vert
\rho\right\vert _{G,p^{\prime}}.$}%
\begin{equation}
\left.
\begin{array}
[c]{c}%
\left\Vert V\right\Vert _{T,p}\leq C\exp\left(  CT(\left(  \left\Vert
\nabla\sigma\right\Vert _{(\mu,\infty)}^{2}+\left\Vert \nabla b\right\Vert
_{\infty}+[\nabla c]_{G,p}^{p}\right)  \right)  \times\\
\times\left(  \left\vert V_{0}\right\vert +\left\Vert H\right\Vert
_{T,p}+\left\Vert h\right\Vert _{T,p}+[\rho]_{G,p}\left\Vert R\right\Vert
_{T,p}\right)  .
\end{array}
\right.  \label{M4}%
\end{equation}

\end{proposition}

\begin{proof}
Let us begin with writing $V_{t}=V_{0}+I_{t}+M_{t}+J_{t},$ where $I_{t}$
designates the integral with respect to $ds$ and so on$.$ Using Burkholder's
inequality
\begin{align*}
\mathbb{E}\left[  \sup_{t\leq T}\left\vert M_{t}\right\vert ^{p}\right]   &
\leq C\mathbb{E}\left[  \left(  \int_{0}^{T}\int_{G}\left(  \left\vert
H(s,z)\right\vert ^{2}+\left\vert \left\langle \nabla\sigma(s,z,X_{s}%
),V_{s}\right\rangle \right\vert ^{2}\right)  \mu(dz)ds\right)  ^{\frac{p}{2}%
}\right] \\
&  \leq C\mathbb{E}\left[  \left(  \int_{0}^{T}\left(  \left\Vert
H(s,\cdot)\right\Vert _{\mathbb{L}^{2}(\mu)}^{2}+\left\Vert \nabla
\sigma\right\Vert _{(\mu,\infty)}^{2}\left\vert V_{s}\right\vert ^{2}\right)
ds\right)  ^{\frac{p}{2}}\right]  .
\end{align*}
H\"{o}lder's inequality then yields%
\[
\left\Vert M\right\Vert _{T,p}\leq C\sqrt{T}(\left\Vert H\right\Vert
_{T,p}+\left\Vert \nabla\sigma\right\Vert _{(\mu,\infty)}\left\Vert
V\right\Vert _{T,p}).
\]
A similar estimate holds true for $I_{t}.$ Let us now give the estimates on
the jump term $J_{t}.$ To shorten notations, we write $dN_{\mu}$ instead of
$N_{\mu}\left(  dsdudz\right)  $ and drop the dependency of the coefficients
on these variables. Moreover, we consider the standard decomposition of
$dN_{\mu}=d\widetilde{N}_{\mu}+d\widehat{N}_{\mu}$ (martingale part and
compensator). Corresponding to this decomposition, we write $J_{t}%
=\widetilde{J}_{t}+\widehat{J}_{t}.$ In order to estimate $\widetilde{J}_{t},$
we will use Burkholder's inequality for jump processes (e.g. \cite[Theorem
2.11]{Kunita2004}) to get%
\begin{align*}
\mathbb{E}\left[  \sup_{t\leq T}\left\vert \widetilde{J}_{t}\right\vert
^{p}\right]   &  \leq C\mathbb{E}\left[  \left(  \int_{0}^{T}\int
_{G\times\lbrack0,2\Gamma]}\left\vert Q+\left\langle \nabla c,V\right\rangle
\right\vert ^{2}1_{\{u\leq\gamma\}}d\widehat{N}_{\mu}\right)  ^{\frac{p}{2}%
}\right] \\
&  +C\mathbb{E}\left[  \int_{0}^{T}\int_{G\times\lbrack0,2\Gamma]}\left\vert
Q+\left\langle \nabla c,V\right\rangle \right\vert ^{p}1_{\{u\leq\gamma
\}}d\widehat{N}_{\mu}\right]  .
\end{align*}
By assumption, one has $\left\vert Q+\left\langle \nabla c,V\right\rangle
\right\vert \leq\left\vert \rho\right\vert \left\vert R\right\vert +\left\vert
\nabla c\right\vert \left\vert V\right\vert .$ Hence (for every fixed time
parameter),
\begin{align*}
\int_{G\times\lbrack0,2\Gamma]}\left\vert Q+\left\langle \nabla
c,V\right\rangle \right\vert ^{2}1_{\{u\leq\gamma\}}d\widehat{N}_{\mu}  &
\leq2\int_{G}\left\vert \rho\right\vert ^{2}\left\vert R\right\vert ^{2}\gamma
d\mu+2\int_{G}\left\vert \nabla c\right\vert ^{2}\left\vert V\right\vert
^{2}\gamma d\mu\\
&  \leq2\left\vert R\right\vert ^{2}\left\vert \rho\right\vert _{G,2}%
^{2}+2\left\vert V\right\vert ^{2}\left\vert \nabla c\right\vert _{G,2}^{2}.
\end{align*}
This leads to the following inequality.
\[
\mathbb{E}\left[  \left(  \int_{0}^{T}\int_{G\times\lbrack0,2\Gamma
]}\left\vert Q+\left\langle \nabla c,V\right\rangle \right\vert ^{2}%
1_{\{u\leq\gamma\}}d\widehat{N}_{\mu}\right)  ^{\frac{p}{2}}\right]  \leq
CT^{\frac{p}{2}}\left(  \left\vert \rho\right\vert _{G,2}^{p}\left\Vert
R\right\Vert _{T,p}^{p}+\left\vert \nabla c\right\vert _{G,2}^{p}\left\Vert
V\right\Vert _{T,p}^{p}\right)  .
\]
In a similar way,%
\[
\mathbb{E}\left[  \int_{0}^{T}\int_{G\times\lbrack0,2\Gamma]}\left\vert
Q+\left\langle \nabla c,V\right\rangle \right\vert ^{p}1_{\{u\leq\gamma
\}}d\widehat{N}_{\mu}\right]  \leq CT(\left\vert \rho\right\vert _{G,p}%
^{p}\left\Vert R\right\Vert _{T,p}^{p}+\left\vert \nabla c\right\vert
_{G,p}^{p}\left\Vert V\right\Vert _{T,p}^{p}).
\]
For the term $\widehat{J}_{t}$, similar arguments yield
\begin{align*}
\mathbb{E}\left[  \sup_{t\leq T}\left\vert \widehat{J}_{t}\right\vert
^{p}\right]   &  \leq\mathbb{E}\left[  \left(  \int_{0}^{T}\int_{G\times
\lbrack0,2\Gamma]}\left\vert Q+\left\langle \nabla c,V\right\rangle
\right\vert 1_{\{u\leq\gamma\}}d\widehat{N}_{\mu}\right)  ^{p}\right] \\
&  \leq CT^{p}\left(  \left\vert \rho\right\vert _{G,1}^{p}\left\Vert
R\right\Vert _{T,p}^{p}+\left\vert \nabla c\right\vert _{G,1}^{p}\left\Vert
V\right\Vert _{T,p}^{p}\right)  .
\end{align*}
Summing up these estimates we conclude that, if $T\leq1,$ then
\[
\left\Vert J\right\Vert _{T,p}\leq CT^{\frac{1}{p}}([\rho]_{G,p}\left\Vert
R\right\Vert _{T,p}+[\nabla c]_{G,p}\left\Vert V\right\Vert _{T,p}).
\]
It follows that
\begin{align}
\left\Vert V\right\Vert _{T,p}  &  \leq\left\Vert V_{0}\right\Vert
_{p}+C([\rho]_{G,p}(\left\Vert R\right\Vert _{T,p}+\left\Vert H\right\Vert
_{T,p}+\left\Vert h\right\Vert _{T,p})\label{M5}\\
&  +C(T^{\frac{1}{2}}\left\Vert \nabla\sigma\right\Vert _{(\mu,\infty
)}+T\left\Vert \nabla b\right\Vert _{\infty}+T^{\frac{1}{p}}[\nabla
c]_{G,p})\left\Vert V\right\Vert _{T,p}.\nonumber
\end{align}
We will use this inequality on the successive intervals $\left(
kT,(k+1)T\right)  ,k\in\mathbb{N}$ for some convenient $T$ (see after) in
order to obtain (\ref{M4}). We take%
\[
T=\min\left\{  \frac{1}{6C\left\Vert \nabla b\right\Vert _{\infty}},\frac
{1}{(6C\left\Vert \nabla\sigma\right\Vert _{(\mu,\infty)})^{2}},\frac
{1}{(6C[\nabla c]_{G,p})^{p}},1\right\}
\]
which implies $C(T^{\frac{1}{2}}\left\Vert \nabla\sigma\right\Vert
_{(\mu,\infty)}+T\left\Vert \nabla b\right\Vert _{\infty}+T^{\frac{1}{p}%
}[\nabla c]_{G,p})\leq\frac{1}{2}.$ Then, the inequality (\ref{M5}) yields%
\[
\left\Vert V\right\Vert _{T,p}\leq2(\left\Vert V_{0}\right\Vert _{p}%
+C([\rho]_{G,p}(\left\Vert R\right\Vert _{T,p}+\left\Vert H\right\Vert
_{T,p}+\left\Vert h\right\Vert _{T,p})).
\]
We denote $Q_{k}=C([\rho]_{G,p}\left\Vert R\right\Vert _{kT,p}+\left\Vert
H\right\Vert _{kT,p}+\left\Vert h\right\Vert _{kT,p})$ and $v_{k}=\left\Vert
V\right\Vert _{kT,p}$ and we obtain%
\[
v_{k+1}\leq2v_{k}+Q_{k}\leq2v_{k}+Q_{n}\quad\forall k\leq n
\]
and as a consequence $v_{n}\leq2^{n}(\left\vert V\right\vert _{0}+Q_{n}).$
Now, let $S$ be fixed and let $n=[S/T]+1.$ Then we get%
\[
\left\Vert V\right\Vert _{S,p}\leq v_{n}\leq2^{n}(\left\vert V\right\vert
_{0}+O_{n})=e^{\left(  [S/T]+1\right)  \ln2}(\left\vert V\right\vert
_{0}+C([\rho]_{G,p}\left\Vert R\right\Vert _{S,p}+\left\Vert H\right\Vert
_{S,p}+\left\Vert h\right\Vert _{S,p})).
\]
We have%
\[
\lbrack S/T]\leq S\times\max\{6C\left\Vert \nabla b\right\Vert _{\infty
},(6C\left\Vert \nabla\sigma\right\Vert _{(\mu,\infty)})^{2},(6C[\nabla
c]_{G,p})^{p}\}
\]
so we conclude.
\end{proof}

The same reasoning based on Burkholder's inequality for jump processes as in
the previous proof leads to the following.

\begin{remark}
For every $p\geq2$ there exists a universal constant $C$ (depending on $p$)
such that for every $f$
\begin{equation}
\left(  \mathbb{E}\left[  \left(  \int_{0}^{t}\int_{G\times(0,2\Gamma
)}\left\vert f(s,z,X_{s-})\right\vert 1_{\{u\leq\gamma_{\Gamma}(s,z,X_{s-}%
)\}}N_{\mu}(ds,du,dz)\right)  ^{p}\right]  \right)  ^{\frac{1}{p}}\leq
C\max\left\{  t,1\right\}  [f]_{G,p}. \label{M3}%
\end{equation}

\end{remark}

\subsubsection{Proof of Lemma \ref{EstimatesPD}%
\label{SectionProofLemmaEstimates}}

\begin{proof}
[Proof of Lemma \ref{EstimatesPD}]We will prove, by recurrence that, for all
$p\geq2k,$
\begin{equation}
\sup_{x\in\mathbb{R}^{d}}\mathbb{E}\left[  \sup_{t\leq T}\left\vert
\partial^{\alpha}\overline{X}_{t}^{G}(x)\right\vert ^{\frac{p}{k}}\right]
^{\frac{k}{p}}\leq\alpha_{k,p}(C,G)=C\theta_{k,p}^{k%
{\textstyle\sum\limits_{1\leq n\leq k}}
\frac{1}{n}}(G)\exp\left(  CTk\left(
{\textstyle\sum\limits_{1\leq n\leq k}}
\frac{1}{n}\right)  a_{p}(G)\right)  . \label{D6'}%
\end{equation}

\textbf{Step 1}. (Chain Estimates) Let $f:\mathbb{R}^{d}\rightarrow\mathbb{R}$
and $g:\mathbb{R}^{d}\rightarrow\mathbb{R}^{d}$ be smooth functions. If
$\left\vert \alpha\right\vert =k\geq1,$ then%
\[
\left\vert \partial^{\alpha}(f(g(x)))\right\vert \leq C\left(  \sum_{i=1}%
^{d}\left\vert (\partial_{i}f)(g(x))\right\vert \times\left\vert
\partial^{\alpha}g^{i}(x)\right\vert +\sum_{2\leq\left\vert \alpha^{\prime
}\right\vert \leq k}\left\vert (\partial^{\alpha^{\prime}}f)\left(
g(x)\right)  \right\vert \times\sum_{1\leq\left\vert \beta\right\vert \leq
k-1}\left\vert \partial^{\beta}g(x)\right\vert ^{\frac{k}{\left\vert
\beta\right\vert }}\right)  .
\]
The above inequality is obtained by taking first derivatives and then by using
Young's inequality in order to separate the different derivatives of $g$. As
an immediate consequence, one gets
\begin{equation}
\left\vert \partial^{\alpha}f\left(  g\left(  x\right)  \right)  \right\vert
\leq C\left(  \left\Vert \nabla f\right\Vert _{\infty}\left\vert
\partial^{\alpha}g(x)\right\vert +\left\Vert f\right\Vert _{2,k,\infty}%
\sum_{1\leq\left\vert \beta\right\vert \leq k-1}\left\vert \partial^{\beta
}g(x)\right\vert ^{\frac{k}{\left\vert \beta\right\vert }}\right)  .
\label{A5}%
\end{equation}
Similar reasoning for $F:\mathbb{R}_{+}\times E\times\mathbb{R}^{d}%
\rightarrow\mathbb{R}$ that is globally measurable and differentiable with
respect to $x\in\mathbb{R}^{d}$ yields
\begin{equation}
\left(  \int_{E}\left\vert \partial_{x}^{\alpha}(F(t,z,g(x))\right\vert
^{2}\mu(dz)\right)  ^{\frac{1}{2}}\leq C\left(  \left\Vert \nabla F\right\Vert
_{(\mu,\infty)}\left\vert \partial^{\alpha}g(x)\right\vert +\left\Vert
F\right\Vert _{2,k,(\mu,\infty)}\sum_{1\leq\left\vert \beta\right\vert \leq
k-1}\left\vert \partial^{\beta}g(x)\right\vert ^{\frac{k}{\left\vert
\beta\right\vert }}\right)  . \label{A5'}%
\end{equation}
Having this inequality in mind (and the notation (\ref{A6})), we introduce the
following notations%
\begin{equation}
y_{\alpha}=\partial^{\alpha}g(x),\quad y_{[k-1]}=(\partial^{\beta}%
g(x))_{1\leq\left\vert \beta\right\vert \leq k-1}. \label{A7'}%
\end{equation}
Using this notation the estimate (\ref{A5}) (resp.(\ref{A5'})) reads
\begin{align}
\left\vert \partial^{\alpha}(f(g(x))\right\vert  &  \leq C(\left\Vert \nabla
f\right\Vert _{\infty}\left\vert y_{\alpha}\right\vert +\left\Vert
f\right\Vert _{2,k,\infty}\left\vert y_{[k-1]}\right\vert _{\mathbb{R}%
_{[k-1]}}),\label{A7}\\
\left(  \int_{E}\left\vert \partial_{x}^{\alpha}(F(t,z,g(x))\right\vert
^{2}\mu(dz)\right)  ^{\frac{1}{2}}  &  \leq C(\left\Vert \nabla F\right\Vert
_{(\mu,\infty)}\left\vert y_{\alpha}\right\vert +\left\Vert F\right\Vert
_{2,k,(\mu,\infty)}\left\vert y_{[k-1]}\right\vert _{\mathbb{R}_{[k-1]}}.
\label{A7''}%
\end{align}

\textbf{Step 2}. (Deriving the Differential Equation for $\partial^{\alpha
}\overline{X}_{t}^{G}(x)$ and Estimates) We denote by $\overline{Y}_{\alpha
}(t,x)=\partial^{\alpha}\overline{X}_{t}^{G}(x)$ and by $\overline{Y}%
_{[k]}(t,x)=(\overline{Y}_{\alpha}(t,x))_{1\leq\left\vert \alpha\right\vert
\leq k}\in\mathbb{R}_{[k]}^{d}$, for all initial data $x.$ We claim that, for
every multi-index $\alpha$ with $\left\vert \alpha\right\vert =k\geq1,$%
\begin{equation}
\left.
\begin{array}
[c]{l}%
\overline{Y}_{\alpha}(t,x)=\partial^{\alpha}\overline{X}_{0}^{G}(x)+\int
_{0}^{t}\left(  g_{\alpha}\left(  s,\overline{X}_{s}^{G},\overline{Y}%
_{[k-1]}(s,x)\right)  +\left\langle \nabla b\left(  s,\overline{X}_{s}%
^{G}\right)  ,\overline{Y}_{\alpha}(s,x)\right\rangle \right)  ds\\
+\int_{0}^{t}\int_{E}\left(  h_{\alpha}\left(  s,z,\overline{X}_{s}%
^{G},\overline{Y}_{[k-1]}(s,x)\right)  +\left\langle \nabla\sigma_{l}\left(
s,z,\overline{X}_{s}^{G}\right)  ,\overline{Y}_{\alpha}(s,x)\right\rangle
\right)  W_{\mu}(ds,dz)\\
+\sum_{j=1}^{J_{t}}\left(  Q_{\alpha}\left(  T_{j},\overline{Z}_{j}%
,\overline{X}_{T_{j}-}^{G}(x),\overline{Y}_{[k-1]}\left(  T_{j}-,x\right)
\right)  +\left\langle \nabla c\left(  T_{j},\overline{Z}_{j},\overline
{X}_{T_{j}-}^{G}\right)  ,\overline{Y}_{T_{j}-}^{\alpha}(x)\right\rangle
\right)  1_{G}(\overline{Z}_{j})
\end{array}
\right.  \label{D9}%
\end{equation}
where $g_{\alpha},h_{\alpha}$ and $Q_{\alpha}$ satisfy
\begin{align}
\left(  \int_{E}\left(  \left\vert h_{\alpha}(t,z,x,y_{[k-1]})\right\vert
^{2}\mu(dz)\right)  \right)  ^{\frac{1}{2}}  &  \leq C\left\Vert
\sigma\right\Vert _{2,k,(\mu,\infty)}\left\vert y_{[k-1]}\right\vert
_{\mathbb{R}_{[k-1]}^{d}}\label{D10a}\\
\left\vert g_{\alpha}(t,x,y)\right\vert  &  \leq C\left\Vert b\right\Vert
_{2,k,\infty}\left\vert y_{[k-1]}\right\vert _{\mathbb{R}_{[k-1]}^{d}}
\label{D10b}%
\end{align}
and%
\begin{equation}
\left\vert Q_{\alpha}(t,z,x,y)\right\vert \leq C\left\vert y_{[k-1]}%
\right\vert _{\mathbb{R}_{[k-1]}^{d}}\sum_{2\leq\left\vert \alpha\right\vert
\leq k}\left\vert \partial_{x}^{\alpha}c(t,z,x)\right\vert . \label{D10c}%
\end{equation}
The equation (\ref{D9}) is obtained by taking formal derivatives in (\ref{D1})
and then, (\ref{D10a}),(\ref{D10b}) and (\ref{D10c}) are obtained by using
(\ref{A7}) and (\ref{A7''}).

\textbf{Step 3}.\ Now we prove (\ref{D6'}) by recurrence on $k.$ If $k=1$ the
inequality (\ref{D6'}) is an immediate consequence of Proposition
\ref{Moments}.\textbf{\ \ }Let us now assume (\ref{D6'}) to hold true for
$k-1.$ In order to prove it for $k$ we will make use of Proposition
\ref{Moments}. A first step is to make use the identity of laws from Lemma
\ref{Law}. We denote by $Y_{[k]}(t,x)=(Y_{\alpha}(t,x))_{1\leq\left\vert
\alpha\right\vert \leq k}\in\mathbb{R}_{[k]}^{d}$ the unique solution of the
system of equations%
\begin{equation}
\left.
\begin{array}
[c]{l}%
Y_{\alpha}(t,x)=\partial^{\alpha}X_{0}^{G}(x)+\int_{0}^{t}\left(  g_{\alpha
}\left(  s,X_{s}^{G},Y_{[k-1]}(s,x)\right)  +\left\langle \nabla_{x}b\left(
s,X_{s}^{G}\right)  ,Y_{\alpha}(s,x)\right\rangle \right)  ds\\
\text{ \ \ \ \ \ \ \ \ \ \ \ \ \ }+\int_{0}^{t}\int_{E}\left(  h_{\alpha
}\left(  s,z,X_{s}^{G},Y_{[k-1]}(s,x)\right)  +\left\langle \nabla_{x}%
\sigma_{l}(s,z,X_{s}^{G}),Y_{\alpha}(s,x)\right\rangle \right)  W_{\mu
}(ds,dz)\\
\text{ \ \ \ \ \ \ \ \ \ \ \ \ \ }+\sum_{j=1}^{J_{t}}(Q_{\alpha}(T_{j}%
,Z_{j},X_{T_{j}-}^{G}(x),Y_{[k-1]}(T_{j}-,x))\\
\text{ \ \ \ \ \ \ \ \ \ \ \ \ \ }+\left\langle \nabla_{x}c(T_{j}%
,Z_{j},X_{T_{j}-}^{G}),Y_{\alpha}(T_{j}-,x)\right\rangle )1_{\{U_{k}\leq
\gamma(T_{j},Z_{j},X_{T_{j}-})}.
\end{array}
\right.  \label{D11}%
\end{equation}
These equations are the same as in (\ref{D9}) but we replace $\overline{X}%
_{s}^{G}$ by $X_{s}^{G},\overline{Z}_{j}$ by $Z_{j}$\ and $1_{G}(\overline
{Z}_{j})$ by $1_{\{U_{k}\leq\gamma(T_{j},Z_{j},X_{T_{j}-})\}}.$ According to
Lemma \ref{Law}, $\overline{Y}_{[k]}(t,x),t\geq0$ has the same law as
$Y_{[k]}(t,x),t\geq0.$ Now we use Proposition \ref{Moments} for $V_{0}%
=\partial^{\alpha}X_{0}^{G}(x)$ (implying that $\left\vert V_{0}\right\vert
\leq1$)$,$ $h(s)=g_{\alpha}(s,X_{s}^{G},Y_{[k-1]}(s,x))$ and%
\[
H(s,z)=h_{\alpha}(s,z,X_{s}^{G}\left(  x\right)  ,Y_{[k-1]}(s,x)),\quad
Q(s,z)=Q_{\alpha}(s,z,X_{s}^{G}(x),Y_{[k-1]}(s,x)),
\]
for all $s\in\left[  0,T\right]  $ and all $z\in E.$ In view of (\ref{D10a})
and (\ref{D10b})
\[
\left\Vert H(s)\right\Vert _{\mathbb{L}^{2}(\mu)}\leq C\left\Vert
\sigma\right\Vert _{2,k,(\mu,\infty)}\times\sup_{s\leq t}\left\vert
Y_{[k-1]}(s,x)\right\vert _{\mathbb{R}_{[k-1]}^{d}},\text{ }\left\vert
h(s)\right\vert \leq C\left\Vert b\right\Vert _{2,k,\infty}\times\sup_{s\leq
t}\left\vert Y_{[k-1]}(s,x)\right\vert _{\mathbb{R}_{[k-1]}^{d}}.
\]
The estimates (\ref{D10c}) give%
\[
\left\vert Q(s,z)\right\vert \leq\rho(s,z,X_{s}^{G}(x))R_{s},\text{ where
}R_{s}=\left\vert Y_{[k-1]}(s-,x)\right\vert _{\mathbb{R}_{[k-1]}^{d}}%
,\quad\rho(s,z,x)=C\sum_{2\leq\left\vert \alpha^{\prime}\right\vert \leq
k}\left\vert \partial_{x}^{\alpha^{\prime}}c(s,z,x)\right\vert .
\]

We recall that $\left\vert \alpha\right\vert =k\geq2$ and $p\geq2k.$ Then, by
Proposition \ref{Moments}, one has%
\begin{align}
\left\Vert \partial^{\alpha}\overline{X}_{\cdot}^{G}(x)\right\Vert
_{T,\frac{p}{k}}  &  =\left\Vert \partial^{\alpha}X_{\cdot}^{G}(x)\right\Vert
_{T,\frac{p}{k}}\label{D6''}\\
&  \leq C\exp(CTa_{\frac{p}{k}}(G))\theta_{k,\frac{p}{k}}(G)\sup
_{x\in\mathbb{R}^{d}}\left(  \mathbb{E}\left[  \sup_{s\leq T}\left\vert
Y_{[k-1]}(s,x)\right\vert _{\mathbb{R}_{[k-1]}^{d}}^{\frac{p}{k}}\right]
\right)  ^{\frac{k}{p}}\vee1\nonumber\\
&  =C\exp(CTa_{\frac{p}{k}}(G))\theta_{k,\frac{p}{k}}(G)\sum_{1\leq\left\vert
\beta\right\vert \leq k-1}\left(  \mathbb{E}\left[  \sup_{s\leq T}\left\vert
\partial^{\beta}X_{s}^{G}(x)\right\vert ^{\frac{kp}{\left\vert \beta
\right\vert k}}\right]  \right)  ^{\frac{k}{p}}.\nonumber
\end{align}
We assume that $1\leq\left\vert \beta\right\vert =r\leq k-1.$\ Using the
recurrence hypothesis and due to the fact that $\frac{kp}{\left\vert
\beta\right\vert k}=\frac{p}{r},$ one gets%
\[
\left(  \mathbb{E}\left[  \sup_{s\leq T}\left\vert \partial^{\beta}X_{s}%
^{G}(x)\right\vert ^{\frac{p}{\left\vert \beta\right\vert }}\right]  \right)
^{\frac{\left\vert \beta\right\vert }{p}}=\left\Vert \partial^{\beta}%
\overline{X}_{\cdot}^{G}(x)\right\Vert _{T,\frac{p}{r}}\leq C\theta_{r,p}^{r%
{\textstyle\sum\limits_{1\leq n\leq r}}
\frac{1}{n}}(G)\exp\left(  CTr\left(
{\textstyle\sum\limits_{1\leq n\leq r}}
\frac{1}{n!}\right)  a_{p}(G)\right)  .
\]
This implies
\[
\left(  \mathbb{E}\left[  \sup_{s\leq T}\left\vert \partial^{\beta}X_{s}%
^{G}(x)\right\vert ^{\frac{kp}{\left\vert \beta\right\vert k}}\right]
\right)  ^{\frac{k}{p}}\leq C\theta_{r,p}^{k%
{\textstyle\sum\limits_{1\leq n\leq k-1}}
\frac{1}{n}}(G)\exp\left(  CTk\left(
{\textstyle\sum\limits_{1\leq n\leq k-1}}
\frac{1}{n}\right)  a_{p}(G)\right)  .
\]
We insert this inequality in (\ref{D6''}) and note that $a_{\frac{p}{k}%
}(G)\leq a_{p}(G)$ and $\theta_{k,\frac{p}{k}}(G)\leq\theta_{k,p}(G)$ to
conclude%
\begin{align*}
\left\Vert \partial^{\alpha}\overline{X}_{\cdot}^{G}(x)\right\Vert
_{T,\frac{p}{k}}  &  \leq C\theta_{k,p}^{1+k%
{\textstyle\sum\limits_{1\leq n\leq k-1}}
\frac{1}{n}}(G)\exp\left(  CTa_{p}(G)\left(  1+k%
{\textstyle\sum\limits_{1\leq n\leq k-1}}
\frac{1}{n}\right)  \right) \\
&  =C\theta_{k,p}^{k%
{\textstyle\sum\limits_{1\leq n\leq k}}
\frac{1}{n}}(G)\exp\left(  CTa_{p}(G)k%
{\textstyle\sum\limits_{1\leq n\leq k}}
\frac{1}{n}\right)  .
\end{align*}
The proof is now complete by taking $pk$ to replace $p$ in (\ref{D6'}).
\end{proof}

\subsubsection{Proof of the Corollary \ref{CorollaryDiff1} and Lemma
\ref{CorollaryDiff2}\label{SectionProofCorDiff}}

We begin with the following simple remark.

\begin{remark}
Whenever $n\in%
\mathbb{N}
^{\ast}$ and $p\leq n,$ one gets%
\[
\left\Vert J_{t}\right\Vert _{p}\leq\left\Vert J_{t}\right\Vert _{n}\text{ and
}\left\Vert J_{t}\right\Vert _{n}^{n}=\frac{d^{n}e^{2\Gamma\mu(G)t\left(
e^{s}-1\right)  }}{ds^{n}}/_{s=0}=P_{n}\left(  \Gamma\mu(G)t\right)  ,
\]
an $n$-degree polynomial. As a consequence, for some large enough constant
depending, eventually, on the upper bound $n$ but not on $\Gamma,$ $\mu(G)$
nor on $t$,
\begin{equation}
\left\Vert J_{t}\right\Vert _{p}\leq C\Gamma\mu(G)\max\left(  t,1\right)  .
\label{EstimJt}%
\end{equation}

\end{remark}

We now give the proof of Corollary \ref{CorollaryDiff1}.

\begin{proof}
[Proof of Corollary \ref{CorollaryDiff1}]In order to prove the first
assertion, one simply writes, ($\mathbb{P-}$almost surely on $k\leq J_{t}),$
\[
\left\vert \partial^{\alpha}\left(  g(T_{k},\overline{X}_{T_{k}-}%
^{G}(x))\right)  \right\vert \leq C\left\Vert g\right\Vert _{1,q,\infty}%
A_{q},\text{ with }A_{q}=1\vee\sum_{1\leq\left\vert \rho\right\vert \leq
q}\sup_{s\leq t}\left\vert \partial_{x}^{\rho}\overline{X}_{s}^{G}%
(x)\right\vert ^{q}.
\]
Using H\"{o}lder's inequality and (\ref{D6}), we upper bound the term in the
left hand side of (\ref{D12}) by%
\begin{align*}
C\left\Vert g\right\Vert _{1,q,\infty}\mathbb{E}\left(  \left[  \left(
J_{t}\times A_{q}\right)  ^{p}\right]  \right)  ^{\frac{1}{p}}  &  \leq
C\left\Vert g\right\Vert _{1,q,\infty}\left\Vert J_{t}\right\Vert
_{\frac{(1+\eta)p}{\eta}}\left(  \mathbb{E}\left[  A_{q}^{(1+\eta)p}\right]
\right)  ^{\frac{1}{(1+\eta)p}}\\
&  \leq C\left\Vert g\right\Vert _{1,q,\infty}\Gamma\mu(G)\max\left(
t,1\right)  \alpha_{q,(1+\eta)pq}^{q}(C,G).
\end{align*}

To prove the second assertion, we write
\[
\sum_{k=1}^{J_{t}}1_{G}(\overline{Z}_{k})\left\vert \partial^{\alpha}g\left(
T_{k},\overline{Z}_{k},\overline{X}_{T_{k}-}^{G}(x)\right)  \right\vert \leq
A_{q}\times B_{q},
\]
with%
\[
B_{q}=\sum_{1\leq\left\vert \beta\right\vert \leq q}B_{q}(\beta)\text{, where
}B_{q}(\beta)=\sum_{k=1}^{J_{t}}1_{G}(\overline{Z}_{k})\left\vert
(\partial^{\beta}g)(T_{k},\overline{Z}_{k},\overline{X}_{T_{k}-}%
^{G}(x))\right\vert .
\]
Using H\"{o}lder's inequality and (\ref{D6}), we upper bound the term in
(\ref{D13}) by%
\[
\left\Vert A_{q}\right\Vert _{\frac{\left(  1+\eta\right)  p}{\eta}}%
\times\left\Vert B_{q}\right\Vert _{(1+\eta)p}\leq\alpha_{q,\frac{(1+\eta
)pq}{\eta}}^{q}(C,G)\times\left\Vert B_{q}\right\Vert _{(1+\eta)p}.
\]
Using the identification of laws from Lemma \ref{Law} and the inequality
(\ref{M3}), one has%
\begin{align*}
&  \left\Vert B_{q}\right\Vert _{(1+\eta)p}=\left\Vert \int_{0}^{t}%
\int_{G\times\lbrack0,2T]}\left\vert \partial^{\beta}g\left(  s,z,X_{s-}%
^{G}(x)\right)  \right\vert 1_{\{u\leq\gamma\left(  s,z,X_{s-}^{G}(x)\right)
\}}N_{\mu}(ds,dz,du)\right\Vert _{(1+\eta)p}\\
&  \leq C[\partial^{\beta}g]_{G,(1+\eta)p}^{(1+\eta)p}.
\end{align*}
The assertion follows by putting these estimates together.
\end{proof}

\subsection{Proofs of Results in Section \ref{Section3Regimes}%
\label{SectionProof3Regimes}}

We begin with Theorem \ref{ThExp}. As already hinted before, the result
follows from Theorem \ref{Conv}.

\begin{proof}
[Proof of Theorem \ref{ThExp}]We use Theorem \ref{Conv} with $k=0$ and $q=3.$
It is easy to check that $P_{t}^{\varepsilon}$ verifies $H_{2}(0)$ and
$H_{3}(0,3)$ (note that the constant $C$ in (\ref{Dist3a}) depends on
$\varepsilon;$ but is not involved in the estimate (\ref{dist6})). And
$\mathcal{P}_{t}$\ verifies $H_{2}(0)$ and $H_{3}(0,3)$ as well. Moreover, the
constant $Q_{3}(t,P)$ defined in (\ref{Dist1}) verifies%
\[
Q_{3}(t,\mathcal{P})\leq C(t\vee1)^{3}C_{\ast}^{42}\exp((t\vee1)CC_{\ast}%
^{36})
\]
where $C_{\ast}$ is the constant in (\ref{r3}) and $C$ is a universal
constant. Moreover, using a Taylor expansion of order three we get
\[
\left\Vert (\mathcal{L}_{\varepsilon}-\mathcal{L})f\right\Vert _{\infty}\leq
C\delta(\varepsilon)\left\Vert f\right\Vert _{3,\infty}%
\]
with $C$ an universal constant. Then (\ref{dist6}) gives (\ref{r6}).
\end{proof}

Next, we proceed with checking the Assumption $H_{1}$ to complete the explicit example.

\begin{proof}
[Proof of Assumption $H_{1}$]We notice that, by the choice of $\alpha,$
$\int_{\varepsilon}^{3\varepsilon}c_{\varepsilon}(z,x)\gamma(x)\frac{dz}%
{z^{2}}=0$ so our equation may be written as
\begin{align*}
X_{t}^{\varepsilon}  &  =x+\int_{0}^{t}\int_{\varepsilon}^{3\varepsilon}%
\int_{0}^{1}c_{\varepsilon}(z,X_{s-}^{\varepsilon})1_{\{u\leq\gamma
(X_{s-}^{\varepsilon})\}}\widetilde{N}_{\mu_{\varepsilon}}(ds,dz,du)\\
&  +\int_{0}^{t}\int_{3\varepsilon}^{1}\int_{0}^{1}c_{\varepsilon}%
(z,X_{s-}^{\varepsilon})1_{\{u\leq\gamma(X_{s-}^{\varepsilon})\}}%
N_{\mu_{\varepsilon}}(ds,dz,du).
\end{align*}
This is the same as the equation (\ref{r1}). We take $E=\left[  0,1\right]  ,$
$A_{\varepsilon}=\left(  0,3\varepsilon\right]  ,B_{\varepsilon}=\left(
3\varepsilon,4\varepsilon\right]  $ and $C_{\varepsilon}=\left(
4\varepsilon,1\right]  $ and we have%
\[
\int_{\varepsilon}^{3\varepsilon}c_{\varepsilon}^{2}(z,x)\gamma(x)\frac
{dz}{z^{2}}=\sigma^{2}(x),\text{ }\int_{3\varepsilon}^{4\varepsilon
}c_{\varepsilon}(z,x)\gamma(x)\frac{dz}{z^{3/2}}=b(x),
\]
so that $\delta_{\sigma}(\varepsilon)=\delta_{b}(\varepsilon)=0.$ Moreover, on
$C_{\varepsilon},$ $c=c_{\varepsilon}$ which implies $\delta_{c,\gamma
}(\varepsilon)=0.$ Finally, simple computations yield $\delta_{A}%
(\varepsilon)+\delta_{B}(\varepsilon)+\delta_{C}(\varepsilon)\leq
C\sqrt{\varepsilon}$ which leads to the desired conclusion.
\end{proof}

\subsection{Proofs of the Results in Section \ref{SectionBoltzmann}%
\label{SectionProofBoltzmann}}

Before proceeding to the proofs, we recall the following estimates for the
derivatives of the above cut off function given in \cite[Lemma 2.3]%
{BallyFournier2011}.

\begin{lemma}
[{\cite[Lemma 2.3]{BallyFournier2011}}]There exists $\varepsilon_{0}>0$ such
that, for every $\varepsilon\in(0,\varepsilon_{0}),$ every multi-index
$\alpha\in\{1,2\}^{l},$ $l\in\mathbb{N}^{\ast}$ and every $v\in\mathbb{R}%
^{2},$\ one has%
\begin{align}
\left\vert \partial^{\alpha}\ln\varphi_{\varepsilon}(\left\vert v\right\vert
)\right\vert  &  \leq C_{l}(1_{\{\left\vert v\right\vert \in(\varepsilon
,\Gamma_{\varepsilon}-1]\}}\left\vert v\right\vert ^{-l}+1_{\{\left\vert
v\right\vert \in(\Gamma_{\varepsilon}-1,\Gamma_{\varepsilon}+1)\}}%
\Gamma_{\varepsilon}^{-1}),\label{bo6a}\\
\left\vert \partial^{\alpha}\varphi_{\varepsilon}^{\kappa}(\left\vert
v\right\vert )\right\vert  &  \leq C_{l}(1_{\{\left\vert v\right\vert
\in(\varepsilon,\Gamma_{\varepsilon}-1]\}}\left\vert v\right\vert ^{\kappa
-l}+1_{\{\left\vert v\right\vert \in(\Gamma_{\varepsilon}-1,\Gamma
_{\varepsilon}+1)\}}\Gamma_{\varepsilon}^{\kappa-1}). \label{bo6b}%
\end{align}
Moreover, for every $\beta\in(0,1],$ $\varepsilon\in(0,\varepsilon_{0}),$ and
$x,y\geq0$%
\begin{equation}
x^{\beta}\left\vert \varphi_{\varepsilon}^{\kappa}(x)-\varphi_{\varepsilon
}^{\kappa}(y)\right\vert \leq C_{\beta}\Gamma_{\varepsilon}^{\kappa}\left\vert
x-y\right\vert ^{\beta}. \label{bo6c}%
\end{equation}

\end{lemma}

\subsubsection{Proof of the First-Order Estimates in Boltzmann Equation}

The proof consists in two major steps. First, we give upper-bounds for the
constants in (\ref{D5'}), (\ref{D5a}) and (\ref{D5''}). Second, we use Theorem
\ref{Conv} for which we check the assumptions. To conclude, we invoke
(\ref{bo8}) together with the estimates provided by Theorem \ref{Conv}.

We recall the parameters associated with $\widehat{\mathcal{P}}^{\delta}$ in
(\ref{D5'}), (\ref{D5a}) and (\ref{D5''}),
\[
\left.
\begin{array}
[c]{l}%
\theta_{q,p,(\delta)}(E_{\delta})=1+\left\Vert b_{\delta}\right\Vert
_{2,q,\infty}+\sum_{2\leq\left\vert \alpha\right\vert \leq q}[\partial
_{v}^{\alpha}c]_{E_{\delta},p}=1+\left\Vert b_{\delta}\right\Vert
_{2,q,\infty},\\
a_{p,\left(  \delta\right)  }(E_{\delta})=\left\Vert \nabla b_{\delta
}\right\Vert _{\infty}+[\nabla c]_{G,p}^{p},\\
\alpha_{q,p,(\delta)}(C,E_{\delta})=C\theta_{q,pq,(\delta)}^{q%
{\textstyle\sum\limits_{1\leq n\leq q}}
\frac{1}{n}}(E_{\delta})\exp\left(  CTq%
{\textstyle\sum\limits_{1\leq n\leq q}}
\frac{1}{n}\left(  \left\Vert \nabla b_{\delta}\right\Vert _{\infty}+[\nabla
c]_{G,pq}^{pq}\right)  \right)  .
\end{array}
\right.
\]
the first expression following from $\partial^{\alpha}c=0$ if $\left\vert
\alpha\right\vert \geq2.$\ 

\begin{lemma}
We assume that $\kappa<\frac{1}{8}$ and we take $q=2$ (so that $q^{2}%
\kappa<\frac{1}{2}).$ For every $a>0$ there exists $\varepsilon_{0}>0$ and
$C\geq1$ such that for every $\varepsilon\in(0,\varepsilon_{0})$ one has%
\begin{equation}
\alpha_{q,4q,(\delta)}(C,E_{\delta})\leq C\varepsilon^{-a}. \label{a1b}%
\end{equation}
Moreover (see (\ref{D14'}) for the notation), for all $r\leq\frac{1-\nu
}{1-\kappa},$
\begin{equation}
\Gamma_{E_{\delta},q}(\gamma)+[\ln\gamma]_{E_{\delta},q,4q}\leq C\times
\delta^{-q\nu}\times\varepsilon^{-q}=C\times\delta^{-q(\nu+r)}. \label{a1c}%
\end{equation}
As a consequence, the constant in the right hand side of (\ref{dist6})
verifies%
\begin{equation}
Q_{q}(T,\widehat{\mathcal{P}}^{\delta})\leq C\times\delta^{-q(\nu+r+a)}.
\label{a1}%
\end{equation}

\end{lemma}

\begin{proof}
(Throughout the proof, $C$ will be a universal real constant and $a$ an
arbitrary small constant that may change from one line to another.)

Since $\left\vert \partial^{\alpha}c(t,\theta,\rho,v)\right\vert
\leq\left\vert \theta\right\vert \times1_{\left\vert \alpha\right\vert =1}$
\ and $\left\vert \gamma\right\vert \leq\Gamma_{\varepsilon}^{\kappa}$\ we
get, for every $p\geq1$%
\[
\int_{E_{\delta}}\left\vert \partial^{\alpha}c(t,\theta,\rho,v)\right\vert
^{p}\gamma(t,\rho,v)\mu(d\theta,d\rho)\leq C\Gamma_{\varepsilon}^{\kappa}.
\]
In particular%
\[
\lbrack\nabla c]_{E_{\delta},4q^{2}}^{4q^{2}}=\sup_{1\leq p\leq4q}\left\vert
\nabla c\right\vert _{E_{\delta},p}^{4q^{2}}\leq C\Gamma_{\varepsilon}%
^{4q^{2}\kappa}%
\]
and consequently, if $4q^{2}\kappa<2$ then, as a consequence of (\ref{bo4''}%
),\ for sufficiently small $\varepsilon>0$%
\[
\exp(CT[\nabla c]_{E_{\delta},4q^{2}}^{4q^{2}})\leq\exp(C\Gamma_{\varepsilon
}^{4q^{2}\kappa})\leq\varepsilon^{-a}.
\]
Moreover, using (\ref{bo6b}) we get%
\[
\left\vert \partial^{\alpha}\left(  c(t,\theta,\rho,v)\gamma_{\varepsilon
}(t,\rho,v)\right)  \right\vert \leq C\left\vert \theta\right\vert
\times(\varepsilon^{1+\kappa-\left\vert \alpha\right\vert }+\Gamma
_{\varepsilon}^{\kappa}).
\]
As a consequence,%
\[
\left\vert \partial^{\alpha}b_{\delta}(t,v)\right\vert \leq C\delta^{1-\nu
}\times\left(  \varepsilon^{1+\kappa-\left\vert \alpha\right\vert }%
+\Gamma_{\varepsilon}^{\kappa}\right)  =C\left[  \delta^{1-\nu+r(1+\kappa
-\left\vert \alpha\right\vert )}+\delta^{1-\nu}\Gamma_{\varepsilon}^{\kappa
}\right]  \leq C
\]
the last inequality being true for $\left\vert \alpha\right\vert =2,$ if
$r\leq\frac{1-\nu}{\left\vert \alpha\right\vert -1-\kappa}=\frac{1-\nu
}{1-\kappa}.$ We infer that $\left\Vert b_{\delta}\right\Vert _{1,q,\infty
}\leq C$ and $\theta_{q,4q^{2}}(E_{\delta})=1+\left\Vert b_{\delta}\right\Vert
_{2,q,\infty}\leq C$ implying (\ref{a1b}).

We now turn to the proof of the inequality (\ref{a1c}). Using (\ref{bo6a}) we
get
\begin{align*}
\gamma\left\vert \partial_{v}^{\alpha}\ln\gamma\right\vert ^{p}  &  \leq
C1_{\{\left\vert v-v_{t}(\rho)\right\vert >\Gamma_{\varepsilon}-1\}}%
+1_{\{\left\vert v-v_{t}(\rho)\right\vert \in\left[  \varepsilon
,\Gamma_{\varepsilon}-1\right]  \}}(1+\left\vert v-v_{t}(\rho)\right\vert
^{-\left\vert \alpha\right\vert })^{p}\left\vert v-v_{t}(\rho)\right\vert
^{\kappa}\\
&  \leq C\left(  1_{\{\left\vert v-v_{t}(\rho)\right\vert >\Gamma
_{\varepsilon}-1\}}+1_{\{\left\vert v-v_{t}(\rho)\right\vert \in\left[
\varepsilon,\Gamma_{\varepsilon}-1\right]  \}}\varepsilon^{-p\left\vert
\alpha\right\vert +\kappa}\right)  .
\end{align*}
For every $p\geq1$ (and small enough $\delta$), we infer%
\[
\left\vert \partial_{v}^{\alpha}\ln\gamma\right\vert _{E_{\delta},p}\leq
C\delta^{-\nu/p}\varepsilon^{-\left\vert \alpha\right\vert }\leq C\delta
^{-\nu}\varepsilon^{-\left\vert \alpha\right\vert },
\]
In particular,
\[
\left.
\begin{array}
[c]{l}%
\lbrack\ln\gamma]_{E_{\delta},4q}=\sum_{1\leq\left\vert \alpha\right\vert \leq
q}\sup_{1\leq p\leq4q}\left\vert \partial_{v}^{\alpha}\ln\gamma\right\vert
_{E_{\delta},4q}\leq C\delta^{-\nu}\varepsilon^{-q}=C\delta^{-\nu-qr}\text{
and }\\
\Gamma_{G,q}(\gamma)=\sum_{h=1}^{q}\sum_{1\leq\left\vert \alpha\right\vert
\leq h}\left\vert \partial_{v}^{\alpha}\ln\gamma\right\vert _{E_{\delta
},h/\left\vert \alpha\right\vert }^{q/\left\vert \alpha\right\vert }\leq
C\delta^{-q\nu}\varepsilon^{-q}=C\delta^{-q(\nu+r)}.
\end{array}
\right.
\]

\end{proof}

\begin{remark}
\label{remq3}The inequality (\ref{a1}) remains true (with the exact same
proof) for $q=3$ and $r\leq\frac{1-\nu}{2-\kappa}$ as soon as $\kappa<\frac
{1}{18}.$
\end{remark}

\begin{proof}
Finally, by gathering these estimates, one obtains (\ref{a1c}).
\end{proof}

\begin{lemma}
\label{LemmaBoltzmann1}Suppose that (\ref{bo4'}) holds and $\kappa<\frac{1}%
{8},$ (we recall that $\varepsilon=\delta^{r}$ with $0<r\leq\frac{1-\nu
}{1-\kappa}).$ For every $a>0$ there exists $\varepsilon_{0}>0$ such that, for
$\varepsilon<\varepsilon_{0}$%
\begin{equation}
\left\Vert \frac{1}{\psi_{2}}(P_{t_{0},t}^{\varepsilon}f-\widehat{P}_{t_{0}%
,t}^{\delta})f\right\Vert _{\infty}\leq C\delta^{2-3\nu-2r-3a}\times\left\Vert
f\right\Vert _{2,\infty}. \label{a2}%
\end{equation}

\end{lemma}

\bigskip

\begin{proof}
We will use Theorem \ref{Conv} with $q=2,k=2$.

\textbf{Step 1}. We check that for every $p\in\mathbb{N}$ and every $a>0,$ one
can find $C\geq1$ and $\varepsilon_{p,a}>0$ such that, for every
$\varepsilon\in(0,\varepsilon_{p,a}),$ the following estimate holds true
\begin{equation}
\mathbb{E}\left[  \left\vert V_{t}^{\varepsilon}(v)\right\vert ^{p}\right]
\leq C\psi_{p}(v)\varepsilon^{-a}. \label{a2a}%
\end{equation}
This implies that the hypothesis $H_{2}(p)$ (see (\ref{Dist2})) holds for the
semigroup $\mathcal{P}^{\varepsilon}$ with $C_{p}(T,\mathcal{P}^{\varepsilon
})=\varepsilon^{-a}.$

To this purpose, we use It\^{o}'s formula for $f_{p}(x)=\left\vert
x\right\vert ^{p}$ to get
\[
\left.
\begin{array}
[c]{l}%
\mathbb{E}\left[  \left\vert V_{t}^{\varepsilon}(v)\right\vert ^{p}\right]
=\left\vert v\right\vert ^{p}+J_{p}(t),\text{ where }\\
J_{p}(t)=\mathbb{E}\left[  \int_{0}^{t}\int_{E\times R_{+}}(\left\vert
V_{s-}^{\varepsilon}+A(\theta)(V_{s-}^{\varepsilon}-v_{s}(\rho))\right\vert
^{p}-\left\vert V_{s-}^{\varepsilon}\right\vert ^{p})1_{\{u\leq\varphi
_{\varepsilon}^{\kappa}(\left\vert V_{s-}^{\varepsilon}-v_{s}(\rho)\right\vert
\}}N(ds,d\theta,d\rho,du)\right]  .
\end{array}
\right.
\]
Using the inequality $\left\vert \left\vert a+b\right\vert ^{p}-\left\vert
a\right\vert ^{p}\right\vert \leq C\left\vert b\right\vert (\left\vert
a\right\vert ^{p-1}+\left\vert b\right\vert ^{p-1})$, we obtain%
\[
\left.
\begin{array}
[c]{l}%
\left\vert J_{p}(t)\right\vert \leq C\Gamma_{\varepsilon}^{\kappa}%
\mathbb{E}\left[  \int_{0}^{t}\int_{-\frac{\pi}{2}}^{\frac{\pi}{2}}\int
_{0}^{1}\left\vert A(\theta)(V_{s}^{\varepsilon}-v_{s}(\rho))\right\vert
\times(\left\vert A(\theta)(V_{s}^{\varepsilon}-v_{s}(\rho))\right\vert
^{p-1}+\left\vert V_{s}^{\varepsilon}\right\vert ^{p-1})\frac{d\theta}%
{\theta^{1+\nu}}d\rho ds\right] \\
\text{ \ \ \ \ \ \ \ }\leq C\Gamma_{\varepsilon}^{\kappa}\mathbb{E}\left[
\int_{0}^{t}\int_{0}^{1}(\left\vert v_{s}(\rho)\right\vert +\left\vert
V_{s}^{\varepsilon}\right\vert )(\left\vert v_{s}(\rho)\right\vert
^{p-1}+\left\vert V_{s}^{\varepsilon}\right\vert ^{p-1})d\rho ds\right] \\
\text{ \ \ \ \ \ \ \ }\leq C\Gamma_{\varepsilon}^{\kappa}\mathbb{E}\left[
\int_{0}^{t}\int_{0}^{1}(\left\vert v_{s}(\rho)\right\vert ^{p}+\left\vert
V_{s}^{\varepsilon}\right\vert ^{p})d\rho ds)\right]  \leq C\Gamma
_{\varepsilon}^{\kappa}(1+E\int_{0}^{t}\left\vert V_{s}^{\varepsilon
}\right\vert ^{p})ds
\end{array}
\right.
\]
The last inequality is a consequence of
\begin{equation}
\int_{0}^{1}\left\vert v_{s}(\rho)\right\vert ^{p}d\rho=\int_{R^{2}}\left\vert
v\right\vert ^{p}f_{s}(dv)\leq C<\infty. \label{a2b}%
\end{equation}
The inequality (\ref{a2a}) is then a consequence of Gronwall's Lemma.

\textbf{Step 2}. Second, we need to estimate, for regular $f,$ the difference
between the actions of infinitesimal operators%
\[
\left(  \mathcal{L}_{t}^{\varepsilon}-\widehat{\mathcal{L}}_{t}^{\delta
}\right)  f(v)=\int_{E_{\delta}^{c}}\mu(d\theta,d\rho)\gamma_{\varepsilon
}(t,\rho,v)(f(v+c(t,\theta,\rho,v))-f(v)-\left\langle \nabla f(v),c(t,\theta
,\rho,v)\right\rangle ).
\]
One easily notes (using, again, (\ref{a2b})), that%
\[
\left\vert \left(  \mathcal{L}_{t}^{\varepsilon}-\widehat{\mathcal{L}}%
_{t}^{\delta}\right)  f(v)\right\vert \leq\left\Vert f\right\Vert _{2,\infty
}\Gamma_{\varepsilon}^{\kappa}\int_{E_{\delta}^{c}}\mu(d\theta,dv)\left\vert
c(t,\theta,\rho,v)\right\vert ^{2}\leq C\psi_{2}(v)\delta^{2-\nu}\left\Vert
f\right\Vert _{2,\infty}\Gamma_{\varepsilon}^{\kappa}.
\]
We conclude that%
\begin{equation}
\left\Vert \frac{1}{\psi_{2}}\left(  \mathcal{L}_{t}^{\varepsilon}%
-\widehat{\mathcal{L}}_{t}^{\delta}\right)  f\right\Vert _{\infty}\leq
C\delta^{2-\nu}\left\Vert f\right\Vert _{2,\infty}\varepsilon^{-a}.
\label{bo11}%
\end{equation}
This proves that (\ref{dist4'}) holds with $k=2$ and $\varepsilon
(t)=\delta^{2-\nu}\varepsilon^{-a}.$

\textbf{Step 3}. We check that $H_{3}(1,2)$ holds true for both $\widehat
{\mathcal{P}}^{\delta}$ and $\mathcal{P}^{\varepsilon}$. We will only check it
for the approximating semigroup $\widehat{\mathcal{P}}^{\delta},$ the
remaining case being very similar. We recall that $f_{t}(dv)=\mathbb{P}%
(V_{t}\in dv)$ where $V_{t}$ is the solution of the equation (\ref{bo3}).
Then, for every $x\in%
\mathbb{R}
^{2},$
\[
b_{\delta}(t,x)=\int_{E_{\delta}^{c}}\mu(d\theta,d\rho)\gamma_{\varepsilon
}(t,\rho,x)c(t,\theta,\rho,x)=\int_{\{\theta\leq\delta\}}\mathbb{E}\left[
A(\theta)(x-V_{t})\varphi_{\varepsilon}^{\kappa}(\left\vert x-V_{t}\right\vert
)\right]  \frac{d\theta}{\theta^{1+\nu}}.
\]
Using (\ref{bo6c}) with $\beta=1$ we get
\[
\left\vert b_{\delta}(t,x)\right\vert +\left\vert \nabla b_{\delta
}(t,x)\right\vert \leq C\Gamma_{\varepsilon}^{\kappa}\psi_{1}(x)
\]
such that
\begin{equation}
\left\Vert \frac{1}{\psi_{1}}\nabla(b_{\delta}(t,\circ)\nabla f)\right\Vert
_{\infty}\leq C\Gamma_{\varepsilon}^{\kappa}\left\Vert f\right\Vert
_{2,\infty}. \label{L1}%
\end{equation}
We write now%
\[
\left.
\begin{array}
[c]{l}%
I_{t}(f)(x):=\int_{E_{\delta}}\mu(d\theta,d\rho)\gamma(t,\rho
,x)(f(x+c(t,\theta,\rho,x))-f(x))\\
\text{ \ \ \ \ \ \ \ \ \ \ }=\int_{E_{\delta}}\mu(d\theta,d\rho)\gamma
(t,\rho,x)\int_{0}^{1}d\lambda\left\langle \nabla f(x+\lambda c(t,\theta
,\rho,x)),c(t,\theta,\rho,x)\right\rangle \\
\text{ \ \ \ \ \ \ \ \ \ \ }=\mathbb{E}\left[  \int_{\{\delta\leq\left\vert
\theta\right\vert \leq\frac{\pi}{2}\}}\frac{d\theta}{\theta^{1+\nu}}%
\varphi_{\varepsilon}^{\kappa}\left(  \left\vert x-V_{t}\right\vert \right)
\int_{0}^{1}d\lambda\left\langle \nabla f(x+\lambda A(\theta)(x-V_{t}%
)),A(\theta)(x-V_{t})\right\rangle \right]  .
\end{array}
\right.
\]
And, using again (\ref{bo6c}) with $\beta=1$, this gives%
\begin{equation}
\left\Vert \frac{1}{\psi_{1}}I_{t}(f)\right\Vert _{\infty}+\left\Vert \frac
{1}{\psi_{1}}\nabla I_{t}(f)\right\Vert _{\infty}\leq C\Gamma_{\varepsilon
}^{\kappa}\left\Vert f\right\Vert _{2,\infty}. \label{L2}%
\end{equation}

\textbf{Step 4}. We use (\ref{dist6}) with $q=2,k=2:$%
\[
\left.
\begin{array}
[c]{l}%
\left\Vert \frac{1}{\psi_{1}}(\mathcal{P}_{t_{0},t}^{\varepsilon}%
f-\widehat{\mathcal{P}}_{t_{0},t}^{\delta})f\right\Vert _{\infty}\leq C\times
C_{1}(t,\mathcal{P}^{\varepsilon})Q_{2}(t,\widehat{\mathcal{P}}^{\delta}%
)\int_{t_{0}}^{t}\varepsilon(s)ds\times\left\Vert f\right\Vert _{2,\infty}\\
\text{ \ \ \ \ \ \ \ \ \ \ \ \ \ \ \ \ \ \ \ \ \ \ \ \ \ \ \ \ \ \ \ \ }\leq
C\varepsilon^{-a}\times\delta^{-2(\nu+r+a)}\times\delta^{2-\nu}\times
\left\Vert f\right\Vert _{2,\infty}=C\delta^{2-3\nu-2r-3a}\times\left\Vert
f\right\Vert _{2,\infty}.
\end{array}
\right.
\]
Here we have used (\ref{a2a}),(\ref{bo11}) and (\ref{a1}). \ 
\end{proof}

We can now provide a proof for Theorem \ref{ThBoltzmannOrder1}.

\begin{proof}
[Proof of Theorem \ref{ThBoltzmannOrder1}]We recall that%
\[
\left.
\begin{array}
[c]{l}%
r=\frac{2-3\nu}{3+\kappa}\text{ and }\varepsilon=\delta^{r}\text{ and we
write}\\
\left\vert \mathbb{E}\left[  f(V_{t})\right]  -\mathbb{E}\left[  f(U^{\delta
}(V_{0})\right]  \right\vert \leq A+B,\text{ where}\\
A=\left\vert \mathbb{E}\left[  f(V_{t})\right]  -\mathbb{E}\left[
f(V_{t}^{\varepsilon}(V_{0})\right]  \right\vert \text{, }B=\left\vert
\mathbb{E}\left[  f(U_{t}^{\delta}(V_{0}))\right]  -\mathbb{E}\left[
f(V_{t}^{\varepsilon}(V_{0})\right]  \right\vert .
\end{array}
\right.
\]
By (\ref{bo8}),
\[
A\leq\varepsilon^{-a}\times\varepsilon^{1+\kappa}\left\Vert f\right\Vert
_{1,\infty}\leq\delta^{r(1+\kappa)-a}\left\Vert f\right\Vert _{1,\infty}.
\]
Since $(2-3\nu)/(1-\nu)\leq3\leq(3+\kappa)/(1-\kappa)$ it follows that
$r\leq(1-\nu)/(1-\kappa)$ and so we may use (\ref{a2}) and we obtain
\[
B\leq\int_{\mathbb{R}^{2}}\left\vert \mathbb{E}\left[  f(V_{t}^{\varepsilon
}(v)\right]  -\mathbb{E}\left[  f(U^{\delta}(v)\right]  \right\vert
f_{0}(dv)\leq C\int_{\mathbb{R}^{2}}(1+\left\vert v\right\vert ^{2}%
)f_{0}(dv)\times\delta^{2-3\nu-2r-3a}\left\Vert f\right\Vert _{2,\infty}.
\]
We conclude that
\[
\left\vert \mathbb{E}\left[  f(V_{t})\right]  -\mathbb{E}\left[  f(U^{\delta
}(V_{0})\right]  \right\vert \leq C\left\Vert f\right\Vert _{2,\infty}%
(\delta^{r(1+\kappa)-a}+\delta^{2-3\nu-2r-3a})\leq C\left\Vert f\right\Vert
_{2,\infty}\delta^{\frac{(2-3\nu)(1+\kappa)}{3+\kappa}-3a}%
\]
the last inequality being a consequence of the choice of $r.$
\end{proof}

\subsubsection{Proof of the Second-Order Estimates in Boltzmann Equation}

We begin with giving some useful estimates for the noise coefficient
$\sigma_{\delta}:$

\begin{lemma}
1. Let $q\in\mathbb{N}^{\ast},$ $r\leq(1-\nu/2)/(q-1-\kappa/2)$ and
$\varepsilon=\delta^{r}.$ Then, the following inequality holds true$\left\Vert
\sigma_{\delta}\right\Vert _{1,q,(\mu,\infty)}\leq C.$

2. $\widehat{\mathcal{L}}_{t}^{\delta}$ verifies $H_{3}(2,3)$\footnote{see
(\ref{Dist3a})}$.$

3\textbf{. }Let us assume that\textbf{\ }$\kappa\leq1/18$ and
\begin{equation}
r\leq\frac{1-\nu}{2-\kappa}\wedge\frac{1-\nu/2}{2-\kappa/2}. \label{L6}%
\end{equation}
Then%
\begin{equation}
\left\Vert \frac{1}{\psi_{3}}(P_{t_{0},t}^{\varepsilon}f-\widehat{P}_{t_{0}%
,t}^{\delta})f\right\Vert _{\infty}\leq C\delta^{3-4\nu-3r-a}\times\left\Vert
f\right\Vert _{3,\infty}. \label{L4}%
\end{equation}

\end{lemma}

\begin{proof}
1.\textbf{ }Using (\ref{bo6b}) (with $\frac{\kappa}{2}$ instead of $\kappa$),
we get%
\[
\left\vert \partial^{\alpha}\sigma(t,\theta,\rho,v)\right\vert =\left\vert
\partial^{\alpha}\left(  c(t,\theta,\rho,v)\gamma^{\frac{1}{2}}(t,\rho
,v)\right)  \right\vert \leq C\left\vert \theta\right\vert \left(
\varepsilon^{1+\frac{\kappa}{2}-\left\vert \alpha\right\vert }+\Gamma
_{\varepsilon}^{\frac{\kappa}{2}}\right)
\]
which gives, for $1\leq\left\vert \alpha\right\vert \leq q,$%
\[
\int_{E_{\delta}^{c}}\left\vert \partial^{\alpha}\sigma(t,\theta
,\rho,v)\right\vert ^{2}\mu(d\theta,d\rho)\leq C\delta^{2-\nu}\left(
\varepsilon^{2+\kappa-2\left\vert \alpha\right\vert }+\Gamma_{\varepsilon
}^{\kappa}\right)  =C\delta^{2-\nu-r(2\left\vert \alpha\right\vert -2-\kappa
)}\leq C
\]
the last inequality being true if $r\leq(2-\nu)/(2q-2-\kappa).$

2. We only check that
\begin{equation}
\left\Vert \frac{1}{\psi_{2}}\nabla(a_{\delta}^{i,j}(t,\circ)\partial
^{i}\partial^{j}f\right\Vert _{\infty}\leq C\Gamma_{\varepsilon}^{\kappa
}\left\Vert f\right\Vert _{3,\infty}. \label{L3}%
\end{equation}
(The remaining estimates are similar to step 3 in the proof of Lemma
\ref{LemmaBoltzmann1}). One has
\[
a_{\delta}^{i,j}(t,v)=\int_{\{\left\vert \theta\right\vert \leq\delta\}}%
\mu(d\theta,d\rho)\left(  c^{i}c^{j}\right)  (t,\theta,\rho,v)\gamma
(t,\rho,v)
\]
so (\ref{L3}) follows from (\ref{bo6c}) with $\beta=1.$

3. We will use Theorem \ref{Conv} with $q=k=3.$ Using the first assertion and
(\ref{a1}) (see Remark \ref{remq3}), we get that both $\widehat{\mathcal{P}%
}_{t}^{\delta}$ and $\mathcal{P}_{t}^{\varepsilon}$ verify $H_{3}(2,3)$ with
$Q_{q}(t,\widehat{\mathcal{P}}^{\delta})\leq C\times\delta^{-q(\nu+r+a)}.$ And
we recall that in (\ref{a2a}) we have proved that $C_{3}(T,\mathcal{P}%
^{\varepsilon})=\varepsilon^{-a}.$ It remains to estimate
\[
(\mathcal{L}_{t}^{\varepsilon}-\widehat{\mathcal{L}}_{t}^{\delta}%
)f(v)=\int_{E_{\delta}^{c}}\mu(d\theta,d\rho)\gamma(t,\rho,v)\left(
\begin{array}
[c]{c}%
f(v+c(t,\theta,\rho,v))-f(v)-\left\langle \nabla f(v),c(t,\theta
,\rho,v)\right\rangle \\
-\frac{1}{2}\sum_{i,j=1}^{d}c^{i}c^{j}(t,\theta,\rho,v)\partial_{ij}^{2}f(v)
\end{array}
\right)
\]
Taylor's formula gives
\[
\left.
\begin{array}
[c]{l}%
\left\vert (\mathcal{L}_{t}^{\varepsilon}-\widehat{\mathcal{L}}_{t}^{\delta
})f(v)\right\vert \leq C\int_{E_{\delta}^{c}}\mu(d\theta,dv)\left\vert
c(t,\theta,\rho,v)\right\vert ^{3}\times\left\Vert f\right\Vert _{3,\infty
}\Gamma_{\varepsilon}^{\kappa}\\
\leq C\delta^{3-\nu}\int_{0}^{1}\left\vert v-v_{t}(\rho)\right\vert ^{3}%
d\rho\times\left\Vert f\right\Vert _{3,\infty}\Gamma_{\varepsilon}^{\kappa
}\leq C\psi_{3}(v)\delta^{3-\nu}\left\Vert f\right\Vert _{3,\infty}%
\Gamma_{\varepsilon}^{\kappa}%
\end{array}
\right.
\]
so that (\ref{dist4'}) holds with $\varepsilon(t)=\delta^{3-\nu}%
\varepsilon^{-a}.$ We use (\ref{dist6}) with $q=k=3$ to get%
\[
\left.
\begin{array}
[c]{c}%
\left\Vert \frac{1}{\psi_{3}}(P_{t_{0},t}^{\varepsilon}f-\widehat{P}_{t_{0}%
,t}^{\delta})f\right\Vert _{\infty}\leq C\times C_{3}(P^{\varepsilon}%
)Q_{3}(\widehat{P}^{\delta})\int_{t_{0}}^{t}\varepsilon(s)ds\times\left\Vert
f\right\Vert _{3,\infty}\\
\leq C\varepsilon^{-a}\times\delta^{-3(\nu+r)}\times\delta^{3-\nu}%
\times\left\Vert f\right\Vert _{3,\infty}=C\delta^{3-4\nu-3r-a}\times
\left\Vert f\right\Vert _{3,\infty}.
\end{array}
\right.
\]

\end{proof}

\begin{proof}
[Proof of Theorem \ref{ThBoltzmann2}]We proceed as in the first-order case by
combining the two errors by taking $r$ such that $r(1+\kappa)=3-4\nu-3r$ which
amounts to $r=\frac{3-4\nu}{4+\kappa}.$ But we need (\ref{L6}) to hold true so
we ask (\ref{L7}) to hold true.
\end{proof}

\bibliographystyle{plain}
\bibliography{biblioPDMPApp}

\end{document}